\documentclass[12pt]{article}

\usepackage[margin=1.5cm, top=1.5cm, bottom=1.5cm]{geometry}

\usepackage[utf8]{inputenc} 
\usepackage[T1]{fontenc}    
\usepackage{url}            
\usepackage{booktabs}       
\usepackage{amsfonts}       
\usepackage{nicefrac}       
\usepackage{microtype}      
\usepackage{mathrsfs}  

\usepackage{amsmath,amssymb,amsthm,xcolor}
\usepackage{array}
\usepackage{float}
\newtheorem{theorem}{Theorem}[section]
\newtheorem{example}[theorem]{Example}

\newtheorem{lemma}[theorem]{Lemma}
\newtheorem{corollary}[theorem]{Corollary}

\newtheorem{proposition}[theorem]{Proposition}
\newtheorem{definition}[theorem]{Definition}

\newtheorem{assumption}[theorem]{Assumption}
\usepackage{graphicx}
\usepackage{subcaption}
\usepackage{mdframed}
\usepackage{lipsum}
\usepackage{wrapfig}
\usepackage[hidelinks,hypertexnames=false]{hyperref}
\usepackage{tikz}
\usepackage{enumitem}
\allowdisplaybreaks
\usepackage{epstopdf}
\usepackage{algorithmic}
\usepackage{algorithm}
\usepackage{booktabs}
\usepackage{changepage}
\usepackage{enumitem}
\usepackage{indentfirst}
\usepackage{stmaryrd} 
\usepackage{soul}

\title{Global convergence of the subgradient method for robust signal recovery}
\author{\large Zesheng Cai\thanks{\url{u3612774@connect.hku.hk}, Department of Mathematics, The University of Hong Kong, Hong Kong.} \and  Lexiao Lai\thanks{\url{lai.lexiao@hku.hk}, Department of Mathematics, The University of Hong Kong, Hong Kong.} \and Tiansheng Li\thanks{\url{u3612747@connect.hku.hk}, Department of Mathematics, The University of Hong Kong, Hong Kong.}}
\date{}

\begin{document}

\maketitle
\vspace*{-5mm}
 \begin{center}
    \textbf{Abstract}
    \end{center}
    \vspace*{-4mm}
 \begin{adjustwidth}{0.2in}{0.2in}
~~~~We study the subgradient method for factorized robust signal recovery problems, including robust PCA, robust phase retrieval, and robust matrix sensing. The resulting objectives are nonsmooth and nonconvex, and can have unbounded sublevel sets, so standard analyses based on descent and coercivity do not apply. For locally Lipschitz semialgebraic objectives, we develop a convergence framework that replaces these requirements with a boundedness condition on continuous-time subgradient trajectories. Under this condition and sufficiently small step sizes of order $1/k$, we show that iterates of the subgradient method remain bounded and the full sequence converges to a critical point. We then verify the required boundedness property for the three robust objectives by adapting existing trajectory analyses, assuming a mild nondegeneracy condition in the matrix sensing case. Finally, for rank-one symmetric robust PCA, we prove that for almost every initialization, the method cannot converge to spurious critical points; consequently, under the same step-size regime, it converges to a global minimum.

\end{adjustwidth}
\section{Introduction}\label{sec:intro}
Robust signal recovery aims to reconstruct an underlying vector or matrix from observations that may contain gross outliers. Such corruptions arise in modern data pipelines due to sensor failures, adversarial perturbations, or heavy tailed noise, and they are classically modeled and analyzed in robust statistics through contamination frameworks and influence function arguments \cite{huber1964robust,hampel1986robust}. In many inverse problems, the outliers are structured: only a small fraction of samples are corrupted, but their magnitudes can be arbitrarily large. A convenient abstraction is
\[
y_i=\phi_i(\theta^\star)+s_i,\qquad i=1,\dots,N,
\]
where $\phi_i$'s represent measurement operators, and $s$ is sparse and captures the gross errors. This perspective motivates $\ell_1$ type losses, which underlie canonical robust estimators based on absolute deviations \cite{koenker1978regression} and connect naturally to sparsity based recovery guarantees in high dimensions \cite{donoho2006compressed,candes2006stable}. We focus on robust recovery models of this form, and we study the behavior of simple first-order methods on the resulting nonsmooth nonconvex objectives. Consider the following examples.

\begin{example}[Robust principal component analysis]\label{exp:rpca}
Given a matrix $M \in \mathbb{R}^{m \times n}$, robust principal component analysis (robust PCA) aims to decompose $M$ into the sum of a low-rank matrix and a sparse matrix. Such models arise in, e.g., recommendation systems \cite{bennett2007netflix}, video surveillance \cite{lrslibrary2015,bouwmans2019}, and face recognition \cite{basri2003lambertian}. In the sparse-corruption abstraction, one assumes
\[
    M = L^\star + S^\star,
\]
where $L^\star$ is low rank and $S^\star$ is sparse. Using a factorization model $L = X Y^\top$ with $X \in \mathbb{R}^{m \times r}$ and $Y \in \mathbb{R}^{n \times r}$ for a rank estimate $r$, a standard robust formulation fits $L$ to $M$ by an entrywise $\ell_1$ loss, leading to
\begin{equation}\label{eq:rpca}
    \min_{X \in \mathbb{R}^{m \times r},\; Y \in \mathbb{R}^{n \times r}} f(X,Y) := \|XY^\top - M\|_1,
\end{equation}
which encourages the residual $XY^\top - M$ to be sparse. This objective has appeared in \cite{ke2005robust,meng2013cyclic,eriksson2010efficient,gillis2018complexity}.
\end{example}

\begin{example}[Robust phase retrieval]\label{exp:phase}
Given measurement vectors $a_1,\dots,a_N \in \mathbb{R}^n$ and observations $b_1,\dots,b_N \in \mathbb{R}$, robust phase retrieval seeks to recover an unknown signal $x^\star \in \mathbb{R}^n$ from quadratic measurements corrupted by sparse outliers. Concretely, we assume
\[
    b_i = |\langle a_i, x^\star \rangle|^2 + s_i, \qquad i = 1,\dots,N,
\]
where $s=(s_i)_{i=1}^N$ is sparse. A standard robust approach minimizes an empirical $\ell_1$ loss, yielding
\begin{equation}\label{eq:phase}
    \min_{x \in \mathbb{R}^n} f(x) 
    := \frac{1}{2N} \sum_{i=1}^N \bigl||\langle a_i, x \rangle|^2 - b_i\bigr|.
\end{equation}
This problem has been studied in \cite{duchi2019solving,davis2020nonsmooth,eldar2014phase}.
\end{example}

\begin{example}[Robust matrix sensing]\label{exp:sensing}
Let $A_1,\dots,A_N \in \mathbb{R}^{m \times n}$ be measurement matrices and let $b_1,\dots,b_N \in \mathbb{R}$ be observations generated from a low-rank matrix $L^\star \in \mathbb{R}^{m \times n}$ with sparse corruptions:
\[
    b_i = \langle A_i, L^\star \rangle + s_i, \qquad i = 1,\dots,N,
\]
where $s=(s_i)_{i=1}^N$ is sparse and $\langle \cdot,\cdot \rangle$ denotes the Frobenius inner product. Using a factorization model $L = X Y^\top$ with $X \in \mathbb{R}^{m \times r}$ and $Y \in \mathbb{R}^{n \times r}$ for a rank parameter $r$, a robust formulation again uses an $\ell_1$ loss:
\begin{equation}\label{eq:sensing}
    \min_{X \in \mathbb{R}^{m \times r},\; Y \in \mathbb{R}^{n \times r}} 
    f(X,Y) := \frac{1}{N} \sum_{i=1}^N \bigl|\langle A_i, X Y^\top \rangle - b_i\bigr|.
\end{equation}
Closely related formulations have appeared in recent studies \cite{li2020nonconvex,charisopoulos2021low,ma2023global}.
\end{example}

All three examples use $\ell_1$ losses to promote sparsity and robustness to sparse gross errors \cite{candes2007sparsity}. They also employ Burer--Monteiro-type factorizations \cite{burer2003} to represent low-rank structure through factors. The resulting objectives in \eqref{eq:rpca}--\eqref{eq:sensing} are therefore nonsmooth, due to absolute values, and nonconvex, due to bilinear parametrizations. Some of these problems admit convex relaxations with recovery guarantees. For example, robust PCA can be formulated as a convex program \cite{candes2011} and solved using augmented Lagrangian and related splitting methods \cite{yuan2013sparse,lin2011}. Nevertheless, the factorized nonconvex formulations are often preferred at scale. They reduce storage by representing rank-$r$ matrices with $O((m+n)r)$ variables, and they avoid repeated large-scale singular value decompositions that can dominate the computational cost of first-order methods when applied to the convex models.

Despite these challenges, the factorized robust objectives are unconstrained and naturally amenable to simple first-order algorithms. In this work, we focus on the subgradient method \cite{shor1962application}, whose iterates satisfy
\begin{equation}\label{eq:update}
    x_{k+1}\in x_k - \alpha_k \partial f(x_k),\quad\forall k\in \mathbb N,
\end{equation}
where $x_0\in \mathbb R^n$ and the step sizes $(\alpha_k)_{k\in \mathbb N}$ are given. Here $\partial f:\mathbb R^n\rightrightarrows \mathbb R^n$ denotes the Clarke subdifferential \cite{clarke1990optimization}, which we define in Section \ref{sec:bounded_traj}. We refer to any realization $(x_k)_{k\in \mathbb N}$ satisfying \eqref{eq:update} as a \emph{subgradient sequence}. Figure \ref{fig:subgradient-all} shows that, on generic instances of Examples \ref{exp:rpca}--\ref{exp:sensing}, subgradient sequences with step sizes on the order of $1/k$ often converge from random initialization. These observations motivate the following question:

\begin{center}
For the objectives in Examples \ref{exp:rpca}--\ref{exp:sensing}, does the subgradient method from an \emph{arbitrary} initialization converge to a critical point under standard step sizes? If so, which critical points does it select?
\end{center}

\begin{figure}[!ht]
    \centering
    \includegraphics[width=\textwidth]{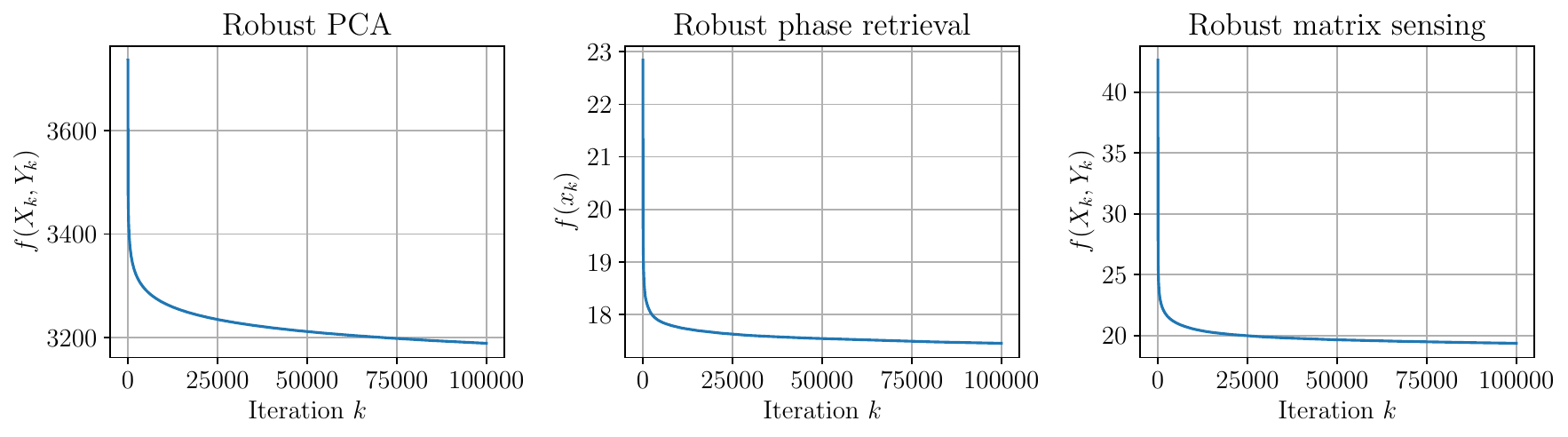}\\[0.6em]
    \includegraphics[width=\textwidth]{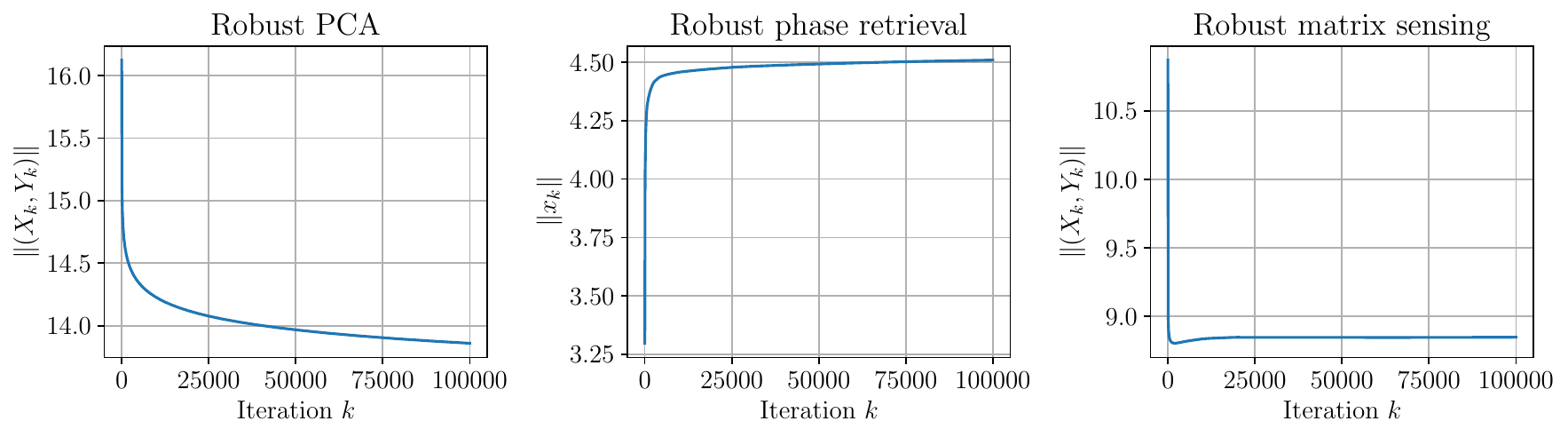}
    \caption{We apply the subgradient method with step size of order $1/k$ to minimize objectives that arise in robust signal recovery problems: the top panel shows the evolution of the objective values, and the bottom panel shows the evolution of the norms of the iterates.}
    \label{fig:subgradient-all}
\end{figure}

 We next review existing convergence guarantees for subgradient methods on the robust recovery problems introduced above. The available theory is predominantly local. For robust phase retrieval \cite{davis2020nonsmooth} and for robust matrix sensing with regularizers \cite{li2020nonconvex}, the subgradient method converges when the initialization lies close enough to a target solution. Such regions are not known to be reached from arbitrary starting points, and spectral initializations that aim to enforce them can be computationally expensive. In contrast, global convergence results for subgradient methods typically require additional structure. For weakly convex objectives, global convergence guarantees are available in \cite{davis2019stochastic2}, and related results for composite models appear in \cite{duchi2018stochastic}. The objectives in Examples \ref{exp:rpca}--\ref{exp:sensing} are weakly convex, but these guarantees do not apply in our setting. Indeed, the analyses in \cite{davis2019stochastic2,duchi2018stochastic} assume a uniform bound on the subgradient norms along the iterates, which effectively requires global Lipschitz control. Such a uniform bound is not available for our factorized formulations, since they are not globally Lipschitz and subgradient norms can grow without bound on unbounded sublevel sets. Locally Lipschitz objectives can instead be treated under bounded-sublevel-set assumptions by using step sizes normalized by subgradient norms \cite{li2024revisiting}. In fact, bounded sublevel sets imply boundedness of subgradient sequences \cite{josz2023globalstability,bolte2025inexact}, for tame functions in the sense of \cite{ioffe2009} (e.g., semialgebraic functions; see Definition \ref{def:semialgebraic}). One then obtains subsequential convergence of the iterates to critical points \cite{davis2020stochastic}. More recently, convergence of the full sequence has been established for step sizes of order $1/k$ \cite{lai2025diameter}. Without coercivity or any a priori boundedness of the iterates, the most general guarantee we are aware of are \cite[Theorems 3.9 \& 3.10]{josz2024proximal}, which establishes convergence to near approximate critical points under carefully chosen summable step sizes. Convergence without bounded iterates has been shown for algorithms enjoying descent-type properties, either in the smooth setting \cite{josz2023global,josz2023convergence} or for proximal algorithms \cite[Theorem 3.11]{josz2024proximal}. The key assumption in these works is boundedness of the associated continuous-time subgradient trajectories; we adopt the same assumption here (Assumption \ref{assumption:bounded}). The proofs then control discrete iterates by tracking the continuous-time trajectories and bounding their length via the total decrease in function values. This descent property, however, is not available for the vanilla subgradient method, whose iterates need not produce monotonically decreasing objective values.

Beyond convergence, it is also crucial to understand \emph{which} critical points the method converges to. Because the objectives \eqref{eq:rpca}--\eqref{eq:sensing} are nonconvex, they typically possess critical points that do not correspond to the desired signal. A comprehensive characterization of the global nonsmooth landscapes in Examples \ref{exp:rpca}--\ref{exp:sensing} remains largely open; to the best of our knowledge, the existing results only concern certain rank-one robust PCA models \cite{joszneurips2018,josz2022nonsmooth,guan2024ell_1,josz2025nonsmooth}. In these settings it is shown that there are no spurious local minima, i.e., every local minimum is globally optimal. This, however, does not preclude the existence of nonglobal critical points that may attract first-order methods, and therefore does not by itself ensure that the subgradient method avoids undesirable solutions. Indeed, as pointed out in \cite[Section 4.1]{guan2024ell_1}, some nonglobal critical points in these models fail to satisfy the hypotheses underpinning current “escape” results for subgradient methods \cite{bianchi2023stochastic,josz2024sufficient,davis2025active}, which rely on partial smoothness of the objective near the critical point \cite{lewis2002active}.

The contribution of this paper is twofold. First, we provide global convergence guarantees for the subgradient method on the factorized robust recovery objectives in Examples \ref{exp:rpca}--\ref{exp:sensing}. We develop a general convergence framework for locally Lipschitz semialgebraic objectives that may have unbounded sublevel sets. Under the assumption that all continuous-time subgradient trajectories are bounded (Assumption \ref{assumption:bounded}), we prove that subgradient sequences with sufficiently small step sizes of order $1/k$ are bounded (Theorem \ref{thm:bounded_diameter}). Together with \cite{lai2025diameter}, this yields convergence of the full subgradient sequence to a critical point (Corollary \ref{convergent corollary}). Building on existing analyses of subgradient trajectories for related models \cite{josz2023certifying,fougereux2024global}, we verify that the objectives in Examples \ref{exp:rpca}--\ref{exp:sensing} satisfy Assumption \ref{assumption:bounded}.

Second, we address critical point selection in a setting where the optimization landscape can be characterized. We study the symmetric variant of the rank-one robust PCA objective \eqref{eq:sym_obj}, which admits spurious critical points. When the data vector has no zero entries, we show that for sufficiently small nonsummable step sizes, subgradient sequences initialized outside a null set cannot converge to these spurious critical points (Theorem \ref{thm:sym_avoid}). Combined with Corollary \ref{convergent corollary}, this implies that with small step sizes of order $1/k$ the subgradient method converges to a global minimum for almost every initialization.

\subsection{Notations}
We recall standard notations used throughout. Let $\mathbb{N}:=\{0,1,2,\dots\}$.
On $\mathbb{R}^n$ we write $\|\cdot\|$ for the norm induced by the Euclidean inner product $\langle\cdot,\cdot\rangle$. For $a\in\mathbb{R}^n$ and $r>0$, let $B(a,r)$ denote the closed ball of center $a$ and radius $r$. If $A\subset\mathbb{R}^n$, then $\overline{A}$ is its closure. We also use the Minkowski enlargement $B(A,r):=A+B(0,r)$. The sign function $\mathrm{sgn}:\mathbb{R}\to\{-1,0,1\}$ returns $-1$ for negative, $1$ for positive, and $0$ at zero.

\section{From bounded subgradient trajectories to converging subgradient sequences}\label{sec:bounded_traj}
Let \( f : \mathbb{R}^n \to \mathbb{R} \) be a locally Lipschitz function. The Clarke subdifferential \cite{clarke1975,clarke1990optimization} is the set-valued mapping \(\partial f : \mathbb{R}^n \rightrightarrows \mathbb{R}^n\) defined for all \( x \in \mathbb{R}^n \) by
\[
\partial f(x) := \left\{ s \in \mathbb{R}^n : f^\circ(x, h) \geq \langle s, h \rangle, \forall h \in \mathbb{R}^n \right\},
\]
where $f^\circ:\mathbb R^n\times \mathbb R^n \to \mathbb R$ denotes the generalized directional derivative defined by
\[
f^\circ(x, h) := \limsup_{\substack{y \to x \\ t \searrow 0}} \frac{f(y + th) - f(y)}{t}.
\]
We say that \( x \in \mathbb{R}^n \) is \textit{critical} if \( 0 \in \partial f (x) \).

To analyze the subgradient method, we will repeatedly relate subgradient sequences to their continuous-time analogue, namely subgradient trajectories. This approach has been used in several recent works on subgradient-type algorithms \cite{davis2020stochastic,bolte2022long,josz2023global,josz2023globalstability}. We say that an absolutely continuous function $x:[0,\infty)\to \mathbb R^n$ is a \emph{subgradient trajectory} of $f$ if it satisfies
\[
x'(t) \in -\partial f(x(t)), \quad \text{for a.e. } t>0.
\]
If $f$ is locally Lipschitz and lower bounded, then it admits subgradient trajectories starting from an arbitrary initial point $x_0$ \cite[Proposition 2.11]{josz2023certifying} (see also \cite[Proposition 2.3]{santambrogio2017euclidean}). In particular, any absolutely continuous $x:[0,T]\to \mathbb R^n$ satisfying the above inclusion over $(0,T)$ can be extended to a solution on $[0,\infty)$. Without additional structure, however, subgradient trajectories can be highly irregular \cite{daniilidis2020pathological}. To obtain useful geometric and analytic properties, we restrict attention to \emph{semialgebraic} objectives \cite{tarski1951decision}, defined as follows.

\begin{definition}
\label{def:semialgebraic}
A subset \( S \) of \( \mathbb{R}^n \) is semialgebraic if it is a finite union of sets of the form
\[
\{x \in \mathbb{R}^n : p_i(x) = 0, \, i = 1, \ldots, k; \, p_i(x) > 0, \, i = k + 1, \ldots, m\}
\]
where \( p_1, \ldots, p_m \) are polynomials defined from \( \mathbb{R}^n \) to \( \mathbb{R} \).
A function \( f : \mathbb{R}^n \to \mathbb{R} \) is semialgebraic if its graph \( \{(x, t) \in \mathbb{R}^{n+1} : f(x) = t\} \) is a semialgebraic set.
\end{definition}

The objectives in Examples \ref{exp:rpca}--\ref{exp:sensing} are all semialgebraic. More broadly, many objectives in data science are semialgebraic whenever they can be expressed as finite compositions of polynomials, $\max/\min$, fractions, and square roots. A key property we will leverage is the chain rule for subgradient trajectories of locally Lipschitz semialgebraic functions \cite{drusvyatskiy2015curves}, recalled next.

\begin{lemma}[{\cite[Lemma 5.2]{davis2020stochastic}}\label{lemma:chain}]
Let $f:\mathbb R^n\to \mathbb R$ be locally Lipschitz semialgebraic. If $x:[0,\infty)\to \mathbb R^n$ is a subgradient trajectory of $f$, then $f\circ x$ is differentiable for almost everywhere on $[0,\infty)$ with
\[
(f\circ x)'(t) = - \|x'(t)\|^2\quad\text{and}\quad\|x'(t)\| = d(0,\partial f(x(t))).
\]
\end{lemma}

We now verify that, under mild assumptions on the problem data, subgradient trajectories are bounded for the robust signal recovery objectives introduced in Section \ref{sec:intro}. This property will later serve as the main standing assumption for our convergence analysis of the subgradient method.

For robust PCA, it was shown in \cite{josz2023certifying} that boundedness holds without any assumption on the data matrix $M$; we restate the result.

\begin{proposition}[{\cite[Proposition 4.5]{josz2023certifying}}]\label{prop:rpca}
Every subgradient trajectory of the robust PCA objective \eqref{eq:rpca} is bounded.
\end{proposition}

The same conclusion holds for robust phase retrieval. Its proof largely follows the boundedness of gradient trajectories of its smooth counterpart \cite{fougereux2024global}, and is thus deferred to Appendix \ref{sec:proof_phase}.

\begin{proposition}\label{prop:phase}
Every subgradient trajectory of the robust phase retrieval objective \eqref{eq:phase} is bounded.
\end{proposition}

For robust matrix sensing (Example \ref{exp:sensing}), we require a standard nondegeneracy condition on the sensing matrices. Specifically, we impose \eqref{eq:rip_lower}, following the same type of assumptions used in \cite[Proposition 4.4]{josz2023certifying} for the smooth counterpart of \eqref{eq:sensing}. Condition \eqref{eq:rip_lower} holds, for instance, when the sensing matrices satisfy the restricted isometry property, which is common in the matrix sensing literature \cite{recht2010guaranteed}. The proof can be found in Appendix \ref{sec:proof_sensing}, and follows the same overall strategy as \cite[Proposition 4.4]{josz2023certifying}. Let $\|\cdot\|_F$ be the Frobenius norm of matrices.

\begin{proposition}\label{prop:sensing}
Let $f:\mathbb R^{m\times r}\times \mathbb R^{n\times r}\to \mathbb R$ be a robust matrix sensing objective \eqref{eq:sensing}, with sensing matrices $A_1,\ldots,A_N\in \mathbb R^{m\times n}$ such that for some $c>0$,
\begin{equation}\label{eq:rip_lower}
\frac{1}{N}\sum_{i = 1}^N \langle A_i,B\rangle^2 \ge c\|B\|_F^2
\end{equation}
for all $B\in \mathbb R^{m\times n}$ with $\mathrm{rank}(B) \le r$. Then every subgradient trajectory of $f$ is bounded.
\end{proposition}

Motivated by Propositions \ref{prop:rpca}--\ref{prop:sensing}, we will establish convergence guarantees for the subgradient method under the following assumption. We identify matrix variables with vectors via vectorization.

\begin{assumption}\label{assumption:bounded}
Let $f:\mathbb{R}^n\to\mathbb{R}$ be locally Lipschitz and semialgebraic. Assume that every subgradient trajectory of $f$ is bounded.
\end{assumption}

\begin{sloppypar}
For semialgebraic functions, it is classical—going back to the pioneering works of \L{}ojasiewicz \cite{lojasiewicz1959,lojasiewicz1982trajectoires} and Kurdyka \cite{kurdyka1998gradients}—that bounded subgradient trajectories must have finite length. Moreover, this length bound can be chosen uniformly over all trajectories initialized in a bounded set \cite[Lemma 1]{josz2023global}, a fact we will use to control the behavior of subgradient sequences.
\end{sloppypar}
\begin{lemma}[{\cite[Lemma 1]{josz2023global}}]\label{lemma:continuous}
Let Assumption \ref{assumption:bounded} hold. Then, for any bounded set $X_0\subset \mathbb{R}^n$, we have
\begin{equation}\label{eq:sigma_traj}
\sigma(X_0) := \sup\left\{\int_{0}^{\infty}\|x'(t)\|\,dt \;\middle|\; \begin{array}{l}
x:[0,\infty)\to\mathbb{R}^n \text{ is absolutely continuous},\\
x'(t)\in -\partial f(x(t)) \text{ for a.e.\ } t>0,\\
x(0)\in X_0
\end{array}\right\}<\infty.
\end{equation}
\end{lemma}

We can now state the main result of this section. It asserts that, with step sizes of order $1/k$, subgradient sequences initialized in a bounded set have uniformly bounded diameters. Recall that the diameter of a subset $A$ of $\mathbb R^n$ is defined by $\mathrm{diam}(A):= \sup\{\|a-b\|:a,b\in A\}$.

\begin{theorem}\label{thm:bounded_diameter}
Let Assumption \ref{assumption:bounded} hold. Then for every bounded $X_0\subset\mathbb{R}^n$ there exist $\bar\alpha>0$ and $c>0$ such that for any subgradient sequence $(x_k)_{k\in\mathbb{N}}$ with $x_0\in X_0$ and any nonincreasing step sizes $(\alpha_k)_{k\in\mathbb{N}}$ with $\alpha_k\le \bar{\alpha}/(k+1)$, 
one has
\(
\mathrm{diam}\bigl((x_k)_{k\in \mathbb N}\bigr)\le c
\).
\end{theorem}

The proof of Theorem \ref{thm:bounded_diameter} is given in Section \ref{sec:proof_diameter}. It relies on recursive arguments similar to those used to prove convergence of the gradient method for smooth functions \cite[Theorem 1]{josz2023global}. In the smooth setting, \cite[Theorem 1]{josz2023global} establishes the stronger property that the gradient-method sequence has uniformly bounded lengths; this property need not hold for subgradient sequences. To address this difficulty, we aim instead for a uniform bound on the diameter of subgradient sequences. This becomes possible by combining the recursive proof with local diameter estimates for subgradient sequences \cite[Theorem 1.1]{lai2025diameter}. 

By \cite[Corollary 1.2]{lai2025diameter}, bounded subgradient sequences with step sizes of order $1/k$ must converge. Combining this fact with Theorem \ref{thm:bounded_diameter} yields the following convergence guarantee.

\begin{corollary}\label{convergent corollary}
Let Assumption \ref{assumption:bounded} hold. Then for every bounded set $X_0 \subset \mathbb{R}^n$, there exists $\bar{\alpha} > 0$ such that for any $\alpha\in (0,\bar{\alpha}]$, it holds that any subgradient sequence $(x_k)_{k \in \mathbb{N}}$ with $x_0 \in X_0$ and step sizes $\alpha_k= \alpha/(k+1)$ converges to a critical point of $f$.
\end{corollary}

Thanks to Propositions \ref{prop:rpca}--\ref{prop:sensing}, Corollary \ref{convergent corollary} applies directly to certify convergence of the subgradient method to critical points for the robust signal recovery problems in Examples \ref{exp:rpca}--\ref{exp:sensing}.

\section{Avoidance of spurious critical points in rank-one symmetric robust PCA}\label{sec:avoidance}
In this section, we study when the subgradient method converges to global minima of robust recovery problems from almost every initialization. This issue is practically important because the objectives \eqref{eq:rpca}--\eqref{eq:sensing} admit spurious critical points, namely critical points that are not global minima. Corollary \ref{convergent corollary} guarantees convergence of subgradient sequences, but it does not exclude convergence to such spurious critical points. To obtain an initialization-independent global guarantee for any optimization algorithm, one would need a complete characterization of the optimization landscape. For nonconvex nonsmooth objectives such as \eqref{eq:rpca}--\eqref{eq:sensing}, this characterization is generally difficult. To the best of our knowledge, the only available landscape results are for rank-one robust PCA \cite{josz2022nonsmooth} and its symmetric variant \cite{joszneurips2018} (see also \cite{guan2024ell_1,josz2025nonsmooth}). Here, we focus on the problem of  rank-one symmetric robust PCA \eqref{eq:sym_obj} and provide conditions under which subgradient sequences converge to global minima for almost every initial point. Such results cannot hold for arbitrary initialization, since the subgradient method cannot avoid spurious critical points when initialized exactly at them.

Throughout this section, we obey the following rule for subscripts: given a vector $x\in \mathbb R^n$ and indices $i,j\in \{1,\ldots,n\}$, we will use $x_i$ and $x_j$ to denote the $i$-th and $j$-th component of $x$ respectively. The subscript $k$ is reserved to present elements in a sequence (e.g., $(x_k)_{k\in \mathbb N}$). The $i$-th component of the vector $x_k$ is denoted by $x_{k,i}$. Consider the objective of rank-one symmetric robust PCA \(f:\mathbb{R}^n \to \mathbb{R}\), defined by
\begin{equation}\label{eq:sym_obj}
    f(x) := \frac{1}{2} \|xx^\top - uu^\top\|_1 = \frac{1}{2} \sum_{i = 1}^n \sum_{j = 1}^n |x_i x_j - u_i u_j|,
\end{equation}
where \(u \in \mathbb{R}^n\) is given. Let $I:= \{1,2,\cdots,n\}$. By \cite[Proposition 1.1]{joszneurips2018} (see also \cite{guan2024ell_1,josz2025nonsmooth}), the critical points of \(f\) are
\[
\underbrace{\{x \in \mathbb{R}^n : \left\langle \mathrm{sign}(u), x \right\rangle = 0 ,  |x_i| \leq |u_i| , \forall i \in I\}}_{A} \cup \{\pm u\},
\]
where \(\pm u\) are global minima and the remaining set $A$ consists of spurious critical points. We call a subset of $\mathbb R^n$ null if it has Lebesgue measure zero. When $u$ has no zero entries, we show that the subgradient method cannot converge to $A$ provided the initialization avoids a null set.

\begin{theorem}\label{thm:sym_avoid}
 Let $f:\mathbb R^n\to \mathbb R$ be the objective of rank-one symmetric robust PCA \eqref{eq:sym_obj} with $u\in \mathbb R^n$ such that $u_i \ne 0$ for all $i\in I$. Then there exists \(\bar{\alpha} > 0\) such that for any nonsummable step sizes \( (\alpha_k)_{k \in \mathbb{N}} \subset (0,\bar{\alpha}]\), there exist a null set \(J \subset \mathbb{R}^n\) such that any subgradient sequence \((x_k)_{k \in \mathbb{N}}\) initialized in \(\mathbb{R}^n \setminus J \) with step sizes \((\alpha_k)_{k \in \mathbb{N}}\) does not converge to any spurious critical point of \(f\).
\end{theorem}

The proof of Theorem~\ref{thm:sym_avoid} is given in Section~\ref{sec:proof_sym} and is based on identifying monotone quantities along the subgradient sequence. The assumption that $u$ has no zero entries is necessary: if some coordinate of $u$ is zero, then there exist initializations from a set of positive Lebesgue measure for which the subgradient method converges to the spurious critical point $0$. This is illustrated in Figure~\ref{fig:subgradient_origin}, where subgradient sequences initialized inside the wedge-shaped region $T$ are drawn toward the origin, in agreement with the conclusion of the proposition below. The proof of Proposition \ref{prop:u_i_zero} is deferred to Section~\ref{sec:proof_u_i_zero}.

\begin{proposition}\label{prop:u_i_zero}
Let $f:\mathbb R^n\to \mathbb R$ be the objective of rank-one symmetric robust PCA \eqref{eq:sym_obj} with $u\in \mathbb R^n\setminus\{0\}$ and $u_1=0$. Fix any nonsummable step sizes $(\alpha_k)_{k\in\mathbb N}\subset (0,1/4]$, and let $(x_k)_{k\in\mathbb N}$ be any subgradient sequence with these step sizes. If
\[
x_0\in T:=\Bigl\{x\in\mathbb R^n:\ |x_i|<\frac{|x_1|}{n+1}\ \ \forall i\in I\setminus\{1\}\Bigr\},
\]
then $x_k\to 0\in A$ as $k\to\infty$.
\end{proposition}
\begin{figure}[ht]
    \centering
    \includegraphics[width=.45\textwidth]{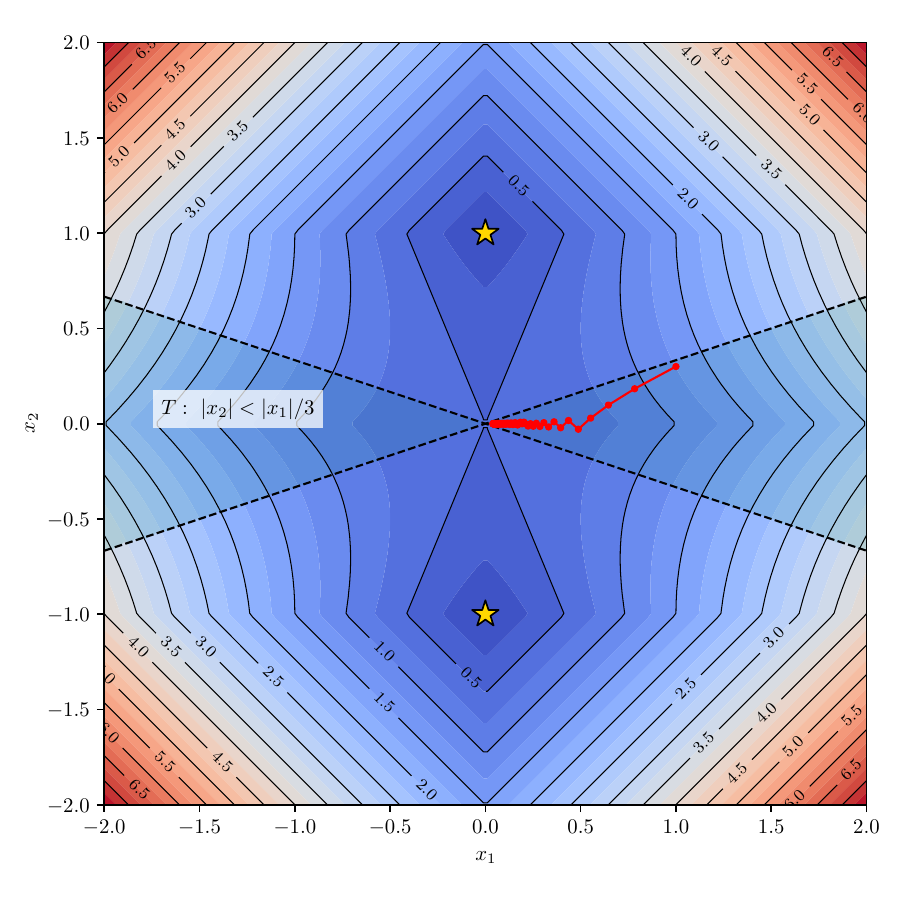}
    \includegraphics[width=.45\textwidth]{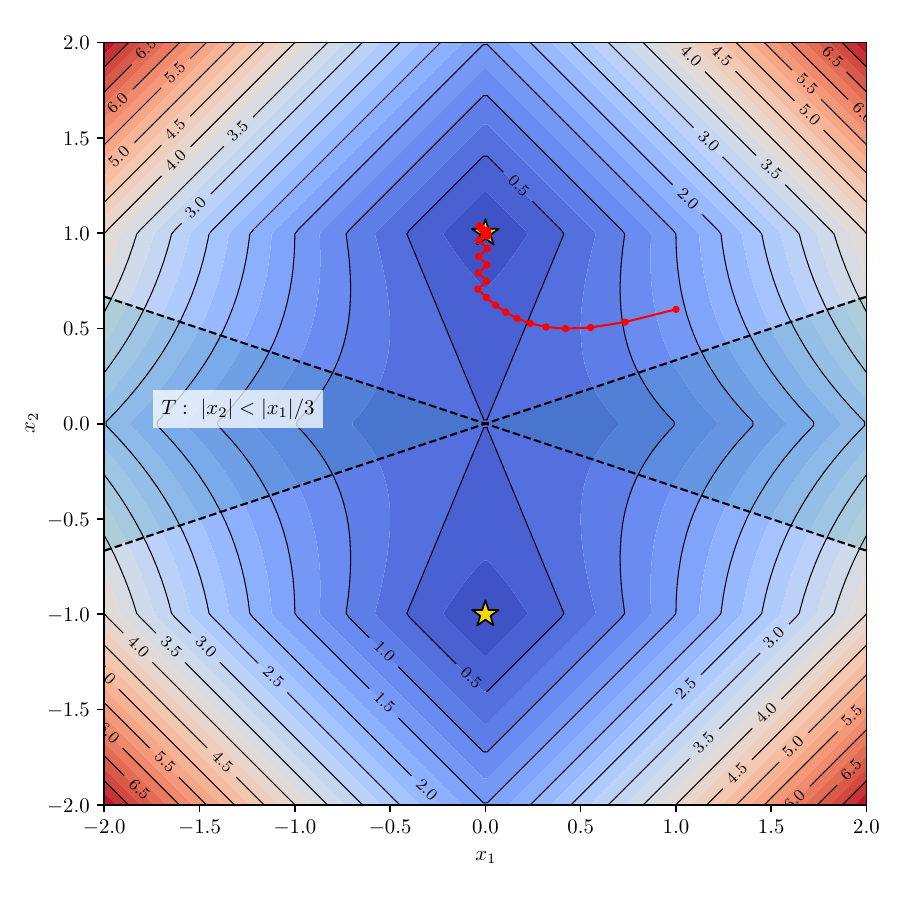}
    \caption{Minimizing $f(x):=\tfrac12\|xx^\top-uu^\top\|_1$ with $u=(0,1)$ by the subgradient method. Red subgradient sequences are initialized at $(1,0.3)$ (left) and $(1,0.6)$ (right) respectively. Stars mark the global minima $\pm u$.}
    \label{fig:subgradient_origin}
\end{figure}

We also note that sublevel sets of the rank-one symmetric robust PCA objective \eqref{eq:sym_obj} are bounded, hence its subgradient trajectories are bounded by Lemma \ref{lemma:chain}. Combining Theorem \ref{thm:sym_avoid} with Corollary \ref{convergent corollary}, we obtain that the subgradient method with sufficiently small step sizes of order $1/k$ converges to a global minimum of $f$ for almost every initialization, as conjectured in \cite[Conjecture 4.5]{guan2024ell_1}.

\begin{corollary}
    Let $f:\mathbb R^n\to \mathbb R$ be the objective of rank-one symmetric robust PCA \eqref{eq:sym_obj} with $u\in \mathbb R^n$ such that $u_i \ne 0$ for all $i\in I$. Then for every bounded $X_0 \subset \mathbb{R}^n$, there exist $\bar{\alpha}> 0$ such that for any $\alpha\in (0,\bar{\alpha}]$, there exists a null set $J\subset \mathbb R^n$ such that any subgradient sequence $(x_k)_{k \in \mathbb{N}}$ with $x_0 \in X_0\setminus J$ and step sizes $\alpha_k= \alpha/(k+1)$ converges to either $u$ or $-u$.
\end{corollary}

\section{Proofs of main results}
This section consists of the proofs of the main results from Sections \ref{sec:bounded_traj} and \ref{sec:avoidance}. 
\subsection{Proof of Theorem \ref{thm:bounded_diameter}}\label{sec:proof_diameter}
We first present two lemmas. The first provides a local upper bound of the diameter of the subgradient sequence with step sizes of order $1/k$, using \cite[Theorem~1.1]{lai2025diameter}.
\begin{lemma}\label{Lemma:diameter}
Let \(f : \mathbb{R}^n \to \mathbb{R}\) be locally Lipschitz and semialgebraic. 
Let \(X \subset \mathbb{R}^n\) be bounded and \(c \in \mathbb{R}\). Then there exist constants \(\bar{\alpha}, \delta, \beta, \epsilon, \varsigma>0\) and 
\(\theta \in (0,1)\) such that the following holds.

Let \((x_k)_{k \in \mathbb{N}}\) be a subgradient sequence with nonincreasing 
step sizes \((\alpha_k)_{k \in \mathbb{N}}\), and let integers \(k,K \in \mathbb{N}\) 
satisfy \(0 \le k \le K\), 
\[
\alpha_m \le \frac{\bar{\alpha}}{m+1} \quad \text{and}
\quad
x_m \in X, \ |f(x_m)-c| \le \epsilon \quad \text{for all } m = k,\dots,K.
\]
Then we have $\delta \,\mathrm{diam}(x_{[k,K]})\le \cdots$
\begin{align*}
&\operatorname{sgn}(f(x_k)-c)\,|f(x_k)-c|^{1-\theta}
   - \operatorname{sgn}(f(x_K)-c)\,|f(x_K)-c|^{1-\theta} \\
&\quad + \varsigma\Bigg(
\left(\frac{\bar{\alpha}}{k+1}\right)^\beta
+ \frac{\bar{\alpha}^{1+\beta}}{\beta}(k+1)^{-\beta}
+ \left(\frac{\bar{\alpha}^{1+\beta}}{\beta}(k+1)^{-\beta}\right)^{1-\theta}
+ \frac{\bar{\alpha}^{1+\theta+\beta\theta}}{\theta\beta^{\theta+1}}(k+1)^{-\beta\theta}
\Bigg).
\end{align*}
\end{lemma}
\begin{proof}
Apply \cite[Theorem~1.1]{lai2025diameter} to the function \(f-c\). 
There exist constants \(\bar{\alpha},\beta,\epsilon,\varsigma_1,\varsigma_2>0\), and 
\(\theta \in (0,1)\) such that the following holds.

Let \((x_k)_{k\in\mathbb{N}}\) be a subgradient sequence with step sizes 
\((\alpha_k)_{k\in\mathbb{N}}\) and let integers \(k,K\) satisfy 
\(0 \le k \le K\), 
\(0 < \alpha_K \le \cdots \le \alpha_k \le \bar{\alpha}\),
and
\(|f(x_t)-c| \le \epsilon\), \(x_t \in X\) for all \(t = k,\dots,K\).
Define \(\psi(s) := \operatorname{sgn}(s)\,|s|^{1-\theta}\). Then
\begin{align*}
\mathrm{diam}(x_{[k,K]})
&\le \varsigma_1\bigl(\psi(f(x_k)-c) - \psi(f(x_K)-c)\bigr) \\
&\quad + \varsigma_2\Bigg(
\alpha_k^{\beta}
+ \sum_{t=k}^K \alpha_t^{1+\beta}
+ \Big(\sum_{t=k}^K \alpha_t^{1+\beta}\Big)^{1-\theta}
+ \sum_{t=k}^K \alpha_t \Big(\sum_{j=t}^K \alpha_j^{1+\beta}\Big)^{\theta}\Bigg).
\end{align*}

By the assumptions of the lemma, the step sizes are nonincreasing and satisfy
\(\alpha_m \le \bar{\alpha}/(m+1)\) for \(m = k,\dots,K\), hence in particular
\(0<\alpha_K \le \cdots \le \alpha_k \le \bar{\alpha}\), so we may apply the
above estimate.

We now bound each series using \(\alpha_t \le \bar{\alpha}/(t+1)\). For each \(t \in \{k,\dots,K\}\),
\[
\sum_{j=t}^K \alpha_j^{1+\beta}
\le \bar{\alpha}^{1+\beta} \sum_{j=t}^{\infty} (j+1)^{-(1+\beta)}
\le \frac{\bar{\alpha}^{1+\beta}}{\beta}(t+1)^{-\beta}.
\]

Therefore,
\begin{align*}
&\alpha_k^{\beta}
+ \sum_{t=k}^K \alpha_t^{1+\beta}
+ \Big(\sum_{t=k}^K \alpha_t^{1+\beta}\Big)^{1-\theta}
+ \sum_{t=k}^K \alpha_t \Big(\sum_{j=t}^K \alpha_j^{1+\beta}\Big)^{\theta} \\
\le~& \left(\frac{\bar{\alpha}}{k+1}\right)^\beta
+ \frac{\bar{\alpha}^{1+\beta}}{\beta}(k+1)^{-\beta}
+ \left(\frac{\bar{\alpha}^{1+\beta}}{\beta}(k+1)^{-\beta}\right)^{1-\theta}
+ \sum_{t=k}^K \frac{\bar{\alpha}}{t+1}
      \left(\frac{\bar{\alpha}^{1+\beta}}{\beta}(t+1)^{-\beta}\right)^{\theta}.
\end{align*}
For the last sum, we obtain
\[
\sum_{t=k}^K \frac{\bar{\alpha}}{t+1}
      \left(\frac{\bar{\alpha}^{1+\beta}}{\beta}(t+1)^{-\beta}\right)^{\theta}
= \frac{\bar{\alpha}^{1+\theta+\beta\theta}}{\beta^{\theta}}
  \sum_{t=k}^K (t+1)^{-(1+\beta\theta)}
\le \frac{\bar{\alpha}^{1+\theta+\beta\theta}}{\theta\beta^{\theta+1}}(k+1)^{-\beta\theta},
\]
where we again compare the sum with an integral over \([k,\infty)\).

Combining the bounds gives $\mathrm{diam}(x_{[k,K]}) \le \cdots$
\begin{align*}
&\varsigma_1\bigl(\psi(f(x_k)-c) - \psi(f(x_K)-c)\bigr) \\
&\quad + \varsigma_2\Bigg(
\left(\frac{\bar{\alpha}}{k+1}\right)^\beta
+ \frac{\bar{\alpha}^{1+\beta}}{\beta}(k+1)^{-\beta}
+ \left(\frac{\bar{\alpha}^{1+\beta}}{\beta}(k+1)^{-\beta}\right)^{1-\theta}
+ \frac{\bar{\alpha}^{1+\theta+\beta\theta}}{\theta\beta^{\theta+1}}(k+1)^{-\beta\theta}
\Bigg).
\end{align*}
Finally, set \(\delta := 1/\varsigma_1\) and \(\varsigma := \varsigma_2/\varsigma_1\) to obtain the claimed inequality.
\end{proof}
The next lemma certifies monotonicity of the first part in the diameter upper bound of Lemma \ref{Lemma:diameter}.

\begin{lemma}\label{lem:increasing-function}
Let \(\theta_0 \in (0,1)\) and set \(\Lambda:=\exp\bigl(\frac{1}{\theta_0-1}\bigr)\).
For \(\lambda_1,\lambda_2\in\mathbb R\) with \(\lambda_1\ge \lambda_2\) and
\(\max\{|\lambda_1|,|\lambda_2|\}\le \Lambda\), define \(g:(0,\theta_0)\to\mathbb R\) by
\[
g(\theta):=\operatorname{sgn}(\lambda_1)|\lambda_1|^{1-\theta}
-\operatorname{sgn}(\lambda_2)|\lambda_2|^{1-\theta}.
\]
Then \(g\) is nondecreasing on \((0,\theta_0)\).
\end{lemma}

\begin{proof}
If \(\lambda_1\lambda_2\le 0\), then \(g(\theta)=|\lambda_1|^{1-\theta}+|\lambda_2|^{1-\theta}\), and since
\(\Lambda<1\) we have \(\ln|\lambda_i|\le \ln\Lambda<0\), hence
\[
g'(\theta)=(-\ln|\lambda_1|)\,|\lambda_1|^{1-\theta}+(-\ln|\lambda_2|)\,|\lambda_2|^{1-\theta}\ge 0,
\]
with the convention that the corresponding term is \(0\) when \(\lambda_i=0\).

Assume now \(\lambda_1,\lambda_2\neq 0\) and \(\lambda_1\lambda_2>0\). For fixed \(\theta\in(0,\theta_0)\), let
\(G_\theta(y):=\ln(y)\,y^{1-\theta}\) for \(y>0\). Then
\[
G_\theta'(y)=y^{-\theta}\bigl(1+(1-\theta)\ln y\bigr).
\]
For \(0<y\le \Lambda\) and \(\theta<\theta_0\),
\[
1+(1-\theta)\ln y \le 1+(1-\theta)\ln\Lambda
=1+\frac{1-\theta}{\theta_0-1}<0,
\]
so \(G_\theta\) is decreasing on \((0,\Lambda]\). If \(\lambda_1,\lambda_2>0\), then
\(g'(\theta)=G_\theta(\lambda_2)-G_\theta(\lambda_1)\ge 0\) since \(\lambda_1\ge \lambda_2\).
If \(\lambda_1,\lambda_2<0\), then \(|\lambda_1|\le |\lambda_2|\) and
\(g'(\theta)=G_\theta(|\lambda_1|)-G_\theta(|\lambda_2|)\ge 0\).
Thus \(g'(\theta)\ge 0\) for all \(\theta\in(0,\theta_0)\), and \(g\) is nondecreasing.
\end{proof}
\begin{proof}[Proof of Theorem \ref{thm:bounded_diameter}]
Let \( X_0 \) be a bounded subset of \( \mathbb{R}^n \), we will show that there exists \( \bar{\alpha} > 0 \) such that \( \sigma(X_0, \bar{\alpha}) < \infty \) where
\begin{equation}
    \label{eq:sigma_X0}
\sigma(X_0, \bar{\alpha}) := 
\sup_{\substack{
    (x_k)_{k\in \mathbb{N}} \in (\mathbb{R}^n)^{\mathbb{N}}}}
\mathrm{diam}((x_k)_{k\in \mathbb{N}})
\end{equation}
subject to
\begin{equation}\label{eq:feasible_sequence}
    \begin{cases}
x_{k+1} \in x_k -\alpha_k \partial f(x_k),\forall k\in \mathbb{N}, \\
    0<\alpha_{k+1} \leq \alpha_k \leq \bar{\alpha}/(1+k),\forall k\in \mathbb{N},\\
x_0 \in X_0.
\end{cases}
\end{equation}
We denote by $\Phi_0$ the collection of all points visited by some subgradient trajectory of \(f\) initialized in \(X_0\). By Lemma \ref{lemma:continuous}, $\Phi_0$ is bounded. Let \(C\) be the set of critical points of \(f\) in \(\overline{\Phi_0}\). By \cite[Corollary 9]{bolte2007clarke}, there are finitely many critical values of $f$ in \(C\), which we denote by $S := \{ f_1, f_2, \dots, f_m \}$.  By Lemma \ref{Lemma:diameter}, for each $f_l \in S$, there exist \(\bar{\alpha}(l), \beta(l), \delta_0(l), \epsilon_0(l),  \varsigma_0(l) > 0\) and \(\theta(l) \in (0, 1)\) such that for any subgradient sequence \((x_k)_{k \in \mathbb{N}}\) with nonincreasing 
step sizes \((\alpha_k)_{k \in \mathbb{N}}\), and let integers \(k,K \in \mathbb{N}\) 
satisfy \(0 \le k \le K\), 
\[
\alpha_m \le \frac{\bar{\alpha}(l)}{m+1} \quad \text{and}
\quad
x_m \in B(\Phi_{0},1), \ |f(x_m)-f_l| \le \epsilon_0(l) \quad \text{for all } m = k,\dots,K.
\]
Then we have
\begin{equation}
    \begin{aligned}
\delta_0(l) \, \mathrm{diam}(x_{[k,K]}) 
&\leq \mathrm{sgn}(f(x_k)-f_l) |f(x_k)-f_l|^{1-\theta(l)} - \mathrm{sgn}(f(x_K)-f_l) |f(x_K)-f_l|^{1-\theta(l)} + \\
&\quad + \varsigma_0(l)\Bigg(\left(\frac{\bar{\alpha}(l)}{k+1}\right)^{\beta(l)}+\frac{\bar{\alpha}(l)^{1+\beta(l)}}{\beta(l)}(k+1)^{-\beta(l)}+\\
&\qquad+\left(\frac{\bar{\alpha}(l)^{1+\beta(l)}}{\beta(l)}(k+1)^{-\beta(l)}\right)^{1-\theta(l)}+\frac{{\bar{\alpha}(l)}^{1+\theta(l)+\beta(l)\theta(l)}}{\theta(l)\beta(l)^{\theta(l)+1}}(k+1)^{-\beta(l)\theta(l)}\Bigg).      \label{eq:diameter bound}     
    \end{aligned}
\end{equation}
We define $\epsilon_0 = \min_{1 \leq l \leq m} \epsilon_0(l)$, $\delta_0 = \min_{1 \leq l \leq m} \delta_0(l)$, $\bar{\alpha}_0 = \min_{1 \leq l \leq m} \bar{\alpha}(l)$, \(\varsigma_0=\max_{1 \leq l \leq m} \varsigma_0(l)\), and $\bar{\theta} = \max_{1 \leq l \leq m} \theta(l)$.
By \cite[2.1.5 Proposition p. 29]{clarke1990optimization} and Lemma \ref{lemma:continuous}, \(C\) is compact. Let $L>1$ so that $f$ is $L$-Lipschitz continuous in $B(C,1)$. Let $(x_k)_{k\in \mathbb N}$ be any subgradient sequence with step sizes $(\alpha_k)_{k\in \mathbb N}$ satisfying the conditions in \eqref{eq:feasible_sequence} for some small $\bar{\alpha}>0$. We next prove that either there exists $k^*\in \mathbb N$ such that $x_{k^*}$ is close to a critical point $x^*\in C$, or $x_k\in B(\Phi_0,1)$ for all $k\in \mathbb N$. Let
\begin{equation}\label{eq:epsilon_def}
\epsilon:= \frac{1}{L} \min\left\{\epsilon_0, \exp(1/(\bar{\theta}-1)),1\right\}
\end{equation}
If \(\overline{\Phi_0}\subset C\), then $x_0 \in X_0 \subset \Phi_0 \subset C$, and we may take $k^* = 0$ and $x^* = x_0$. Otherwise, since \(f\) is continuous and $C$ is compact, there exists \(\delta \in \left(0,\epsilon/2\right)\) such that $\overline{\Phi_0} \setminus \overset{\circ}{B}(C, \delta/3)  \neq \emptyset$ and
\begin{equation}\label{eq:psi_small}
\psi(\lambda):= \mathrm{sgn}(\lambda) |\lambda|^{1-\bar{\theta}} \leq\frac{\delta_0\epsilon}{8}, \quad \forall \lambda \in (-\delta L,\delta L).
\end{equation}

We will show either there exists $k^*\in \mathbb N$ such that $x_{k^*}\in B(C,\delta)$, or $x_k\in B(\Phi_0,1)$ for all $k\in \mathbb N$. Let
\[
\mu:= \inf\bigl\{d(0,\partial f(x)) : x\in \overline{\Phi_0} \setminus B^\circ(C, \delta/3)\bigr\}.
\]
The set \(\overline{\Phi_0} \setminus B^\circ(C, \delta/3)\) is nonempty and compact, and it contains no critical point of \(f\) by the definition of \(C\). By upper semicontinuity of \(\partial f\) \cite[2.1.5 Proposition p.~29]{clarke1990optimization}, the function \(x\mapsto d(0,\partial f(x))\) is lower semicontinuous, so it attains its minimum \(\mu\) on this set. Since there is no \(x\) with \(0\in\partial f(x)\) there, we have \(\mu>0\).

Thus, we may define \( T := 2\sigma(X_0)/\mu \), where \(\sigma(X_0)\) is the uniform upper bound of lengths of subgradient trajectories as defined in \eqref{eq:sigma_traj}, which is finite by Lemma \ref{lemma:continuous}. Since \(\overline{\Phi_0} \not\subset C\), we know \(\sigma(X_0) > 0\). Then for any subgradient trajectory \(x(\cdot)\) initialized in \(X_0\), there exists \(t^*\in (0,T)\) such that \(d(0,\partial f(x(t^*)))< \mu\). Indeed, if we assumed \(d(0,\partial f(x(t))) \ge \mu\) for almost every \(t\in (0,T)\), then, using by Lemma \ref{lemma:chain}, we would obtain
\[
\sigma(X_0) < T\frac{2\sigma(X_0)}{T} \leq \int_{0}^{T} d(0,\partial f(x(t)))\,dt  = \int_{0}^{T} \|x'(t)\|\,dt 
\leq \int_{0}^{\infty} \|x'(t)\|\,dt \leq \sigma(X_0),
\]
a contradiction.
Since \(x(t^*) \in \Phi_0\subset \overline{\Phi_0}\) and \(d(0,\partial f(x(t^*)))<\mu\), we must have \(x(t^*)\in B^\circ(C,\delta/3)\) by the definition of \(\mu\). Hence there exists \(x^*\in C\) such that \(\|x(t^*) - x^*\|\le \delta/3\).

By \cite[Lemma 4.1]{josz2024proximal}, there exists $\bar{\alpha}\in (0,\min\{\bar{\alpha}_0,\delta/(3L)\}$ such that for any subgradient sequence $(x_k)_{k\in \mathbb N}$ satisfying the conditions in \eqref{eq:feasible_sequence}, there exists a subgradient trajectory \( x \colon [0, \infty) \to \mathbb{R}^n \) such that \(
x(0) \in X_0 \) and
\[
\forall k \in \mathbb{N}, \quad \alpha_0 + \cdots + \alpha_{k-1} \leq T \implies \|x_k - x(\alpha_0 + \cdots + \alpha_{k-1})\| \leq \frac{\delta}{3}.
\]
If $\sum_{k=0}^{\infty} \alpha_k \leq T$, then we have $x_k \in B(\Phi_0, \delta/3)$ for all $k \in \mathbb{N}$. Otherwise, since $\alpha_k \leq \delta/(3L)$ for all $k \in \mathbb{N}$, there exists $k^* \in \mathbb{N}$ such that $t_{k^*} := \sum_{k=0}^{k^*-1} \alpha_k \in [t^* - \delta/(3L), t^*]$. So  \(x_m \in B(\Phi_{0},\delta/3)\subset B(\Phi_{0},1)\) for $m = 0,\ldots,k^*$ and
\begin{align*}
\|x_{k^*} - x^*\|\leq~& \|x_{k^*} - x(t_{k^*})\| + \|x(t_{k^*}) - x(t^*)\| + \|x(t^*) - x^*\| \\
\leq~&\delta/3 + L|t_{k^*} - t^*| + \delta/3\\
\leq ~&\delta/3 + \delta/3 + \delta/3 = \delta.
\end{align*}  
Let $K := \inf\{k \geq k^* : x_k \notin B(x^*, \epsilon)\}$. If $K<\infty$, then \(x_l \in B(x^*,\epsilon)\subset B(\Phi_{0},1)\) for \(k^*+1\leq\ l\leq K-1\). From \eqref{eq:diameter bound}, we have
\begin{align*}
\delta_0 \, \mathrm{diam}(x_{k\in [k^*,K-1]}) 
\leq \psi_*(f(x_{k^*})-f_{l^*}) - \psi_*(f(x_{K-1})-f_{l^*}) + h(\bar{\alpha}),
\end{align*}
where $\psi_*:\mathbb R\to \mathbb R$ is given by $\psi_*(\lambda):= \mathrm{sgn}(\lambda)|\lambda|^{1-\theta(l^*)}$ and \(h:[0,\infty)\to[0,\infty)\) is defined by
\[
h(\bar{\alpha})
:= \varsigma_0 \max_{1\le l\le m} \Bigg[
\bar{\alpha}^{\beta(l)}
+ \frac{\bar{\alpha}^{1+\beta(l)}}{\beta(l)}
+ \left(\frac{\bar{\alpha}^{1+\beta(l)}}{\beta(l)}\right)^{1-\theta(l)}
+ \frac{\bar{\alpha}^{1+\theta(l)+\beta(l)\theta(l)}}{\theta(l)\beta(l)^{\theta(l)+1}}
\Bigg].
\]
After possibly reducing $\bar{\alpha}$, we may assume that $h(\bar{\alpha})+\delta_0L\bar{\alpha}<\delta_0\epsilon/4$. Since $\delta\le \epsilon/2$, it holds that
\begin{align*}
\psi_*(f(x_{k^*})-f_{l^*}) - \psi_*(f(x_{K-1})-f_{l^*})
&\ge \delta_0 \, \mathrm{diam}(x_{k\in [k^*,K-1]})-h(\bar{\alpha})\\
&\ge \delta_0\|x_{k^*}-x_{K-1}\|-h(\bar{\alpha})\\
&\ge \delta_0\bigl(\|x_K-x^*\|-\|x_{k^*}-x^*\| - \|x_{K}-x_{K-1}\|\bigr)-h(\bar{\alpha})\\
&\ge \delta_0(\epsilon-\delta - L\bar{\alpha})-h(\bar{\alpha})\\
&\ge \frac{\delta_0\epsilon}{2}-h(\bar{\alpha})-\delta_0 L\bar{\alpha}\\
&> \frac{\delta_0\epsilon}{4}>0.
\end{align*}
Since \(\psi_*\) is increasing, we have $f(x_{k^*})>f(x_{K-1})$. Furthermore,
\[\max\left\{|f(x_{k^*})-f_{l^*}|,|f(x_{K-1})-f_{l^*}|\right\}\le L\epsilon \le \exp\left(\frac{1}{\bar{\theta}-1}\right).\]
by the definition of $\epsilon$ in \eqref{eq:epsilon_def}. Thus, by Lemma \ref{lem:increasing-function} and $\theta(l^*)\in (0,\bar{\theta}]$, we have
\[\psi(f(x_{k^*})-f_{l^*}) - \psi(f(x_{K-1})-f_{l^*}) \ge \psi_*(f(x_{k^*})-f_{l^*}) - \psi_*(f(x_{K-1})-f_{l^*}) \ge \frac{\delta_0\epsilon}{4}.\]
Rearranging and invoking \eqref{eq:psi_small} yields
\begin{align*}
\psi(f(x_{K-1})-f(x^*))&\leq \psi(f(x_{k^*})-f(x^*))-\frac{\delta_0\epsilon}{4}\\
&\leq \frac{\delta_0\epsilon}{8}-\frac{\delta_0\epsilon}{4}=-\frac{\delta_0\epsilon}{8}.
\end{align*}
Since $f(x^*)\leq\max_C f(x)$ and $\psi$ is increasing, we have
\[
x_{K-1}\in X_1 = B(C, \epsilon) \bigcap \left\{ x \in \mathbb{R}^n : f(x)\leq\max_C f(x) -\psi^{-1}\left(\frac{\delta_0\epsilon}{8}\right)\right\}.
\]
By the definition of \(\sigma(\cdot, \cdot)\) in \eqref{eq:sigma_X0}, it holds that
\begin{align*}   
\mathrm{diam}((x_k)_{k\in \mathbb{N}}) &\leq \mathrm{diam}((x_k)_{k\in [0,K-1]})+ \mathrm{diam}((x_k)_{k\in [K-1,\infty)})\\
&\leq \mathrm{diam}(\overline{\Phi_0})+2 +\max\bigl\{0, \sigma(X_1, \bar{\alpha})\bigr\}.
\end{align*}
Note that the inequality still holds if $K = \infty$ or $\sum_{k = 0}^{\infty}\alpha_{k} \le T$. Hence,
\[
\sigma(X_{0},\bar{\alpha})\leq \mathrm{diam}(\overline{\Phi_0})+2 +\max\bigl\{0, \sigma(X_1, \bar{\alpha})\bigr\}.
\]
It now suffices to replace \( X_0 \) by \( X_1 \) and repeat the proof starting from \eqref{eq:sigma_X0}. Denote by $\Phi_1$ the collection of points visited by subgradient trajectories initialized in $X_1$, then
\[\sup_{x\in \Phi_1}f(x) = \sup_{x\in \overline{\Phi_1}}f(x) \le \sup_{x\in X_1}f(x)<\sup_{x\in C}f(x).\]
Thus, the maximal critical value of \( f \) in \( \overline{\Phi_1}\) is less than the maximal critical value of \( f \) in \(\overline{\Phi_0} \). By the definable Morse-Sard theorem \cite[Corollary 9]{bolte2007clarke}, \( f \) has finitely many critical values. Thus, it is eventually the case that one of the sets \( X_0, X_1, \ldots \) is empty. In order to conclude, one simply needs to choose an upper bound on the step sizes \( \hat{\alpha} \) corresponding to \( X_1 \) that is less than or equal to the upper bound \( \bar{\alpha} \) used for \( X_0 \). \( \sigma(X_0, \cdot) \) is finite when evaluated at the last upper bound thus obtained.
\end{proof}
\subsection{Proof of Theorem \ref{thm:sym_avoid}}\label{sec:proof_sym}
We will use the following fact that subgradient sequences do not visit any given null set, for almost every initialization. This result has appeared in the literature \cite{bolte2020mathematical,bianchi2022convergence} (see also \cite[Lemma 5.9]{lai2024stability}). The proof of Lemma \ref{lemma:avoidance} follows similarly as the proof of \cite[Lemma 5.9]{lai2024stability}, and is thus omitted.
\begin{lemma}\label{lemma:avoidance}
Let \(f:\mathbb{R}^{n}\to\mathbb{R}\) be a locally Lipschitz semialgebraic function. There exists a finite set \(F = \{a_{1},\ldots,a_{m}\} \subset (0,\infty)\) such that for any sequence of step sizes \((\alpha_k)_{k\in \mathbb N}\) with \(\alpha_k \notin F\) for all \(k\), and for any null set \(S\subset\mathbb{R}^{n}\), there exists a null set \(J\subset\mathbb{R}^{n}\) such that, for every subgradient sequence $(x_k)_{k\in \mathbb N}$ with step sizes \((\alpha_k)_{k\in \mathbb N}\) and \(x_0 \in \mathbb{R}^{n}\setminus J\), we have $x_k\not\in S$ for all $k\in \mathbb N$.
\end{lemma}

To prove Theorem \ref{thm:sym_avoid}, we consider the set of points where $f$ is nonsmooth, namely, \[P := \{x \in   \mathbb{R}^n:x_i x_j = u_i u_j~~\text{for some} ~~i,j\in I\}.\] Clearly \(P\) is a null set. By Lemma \ref{lemma:avoidance}, there exists $\bar{\alpha}>0$ such that for any \((\alpha_k)_{k\in \mathbb{N}}\subset (0,\bar{\alpha}]\), there exists a null set \(J\subset \mathbb R^n\) such that any subgradient sequence $(x_k)_{k\in \mathbb N}$ with step sizes $(\alpha_k)_{k\in \mathbb N}$ and initialized outside $J$ satisfies $x_k\not\in P\cup A$ for all $k\in \mathbb N$. Thus, any such subgradient sequence follows
\begin{equation}\label{eq:update_simplified}
    x_{k+1}= x_k - \alpha_k\partial f(x_k) = x_k - \alpha_k\nabla f(x_k)~~,~~\forall k\in \mathbb{N}.
\end{equation}

Fix any step sizes  \((\alpha_k)_{k\in \mathbb{N}}\subset (0,\bar{\alpha}]\) and the corresponding null set $J\subset \mathbb R^n$. We next prove Theorem \ref{thm:sym_avoid} by contradiction and let \((x_k)_{k \in \mathbb{N}}\) be a subgradient sequence with step sizes \((\alpha_k)_{k\in \mathbb{N}}\) such that $(x_k)_{k\in \mathbb N}\subset \mathbb R^n\setminus P$ and \(x_k \to x^* \in A\). Consider the following partition of indices \(S_+ := \{i \in I : x^*_i = u_i\}\), \(S_- := \{i \in I : x^*_i = -u_i\}\), and $S_0:= I \setminus (S_+ \cup S_-)$. 
If $S_+ = S_- = \emptyset$, then $|x^*_i|<|u_i|$ for all $i\in I$. Thus, $f$ agrees with the (smooth) quadratic $x\mapsto \frac{1}{2}\sum_{i = 1}^n\sum_{j = 1}^n \mathrm{sign}(-u_i u_j) (x_ix_j - u_i u_j)$ near $x^*$, for which $x^*$ is a maximizer. By the definition of $J$, $x_k \ne x^*$ for all $k$, thus $x_k\not\to x^*$ after possibly reducing $\bar{\alpha}$. 

We next consider the case where \(S_+ \cup S_- \neq \emptyset \). Let $\epsilon>0$ so that for all $x\in B(x^*,\epsilon)$, the following holds:
\begin{enumerate}
    \item $\mathrm{sign}(x_i) = \mathrm{sign}(u_i)$ for all $i\in S_+$;
    \item $\mathrm{sign}(x_i) =- \mathrm{sign}(u_i)$ for all $i\in S_-$;
    \item $\mathrm{sign}(x_i x_j - u_i u _j) =- \mathrm{sign}(u_i u_j)$ for all $i\in I$ and $j\in S_0$.
\end{enumerate}
Indeed, such $\epsilon$ exists by the definition of $S_\pm$, and the fact that $|x_j^*|<|u_j|$ for all $j\in S_0$. Since \(x_k \to x^*\), there exists \(N_1\in \mathbb N\) such that \(x_k \in B(x^*,\epsilon) \setminus P\) for all \(k \ge N_1\). By chain rule \cite[2.3.10 Chain Rule II]{clarke1990optimization}, The partial derivatives of $f$ are given by
\begin{equation}\label{eq:partial_exp}
    \partial_i f(x) = \sum_{j=1}^n\mathrm{sign}(x_ix_j-u_iu_j)x_j,
\end{equation}
where $\mathrm{sign}$ denotes the elementwise set-valued sign mapping defined by $\mathrm{sign}(t) := t/|t|$ if $t\neq 0$, and $\mathrm{sign}(0) := \left[-1,1\right]$. The following lemma concerns the partial derivatives of $f$ near $x^*$, and is essential for the construction of monotone quantities for the proof of Theorem \ref{thm:sym_avoid}.
\begin{lemma}\label{lemma:partial}
 There exists $\delta>0$ such that for any \(x \in B(x^*,\epsilon) \setminus P\), \(I_+ \subset I \setminus S_-\), and \(I_- \subset I \setminus S_+\), we have
    \[
    \frac{1}{|I_+|}\sum_{i_+ \in I_+}\mathrm{sign}(u_{i_+})\partial_{i_+} f(x) - \frac{1}{|I_-|}\sum_{i_- \in I_-}\mathrm{sign}(u_{i_-})\partial_{i_-} f(x) \geq 0.
    \]
    Furthermore, it holds that either \(|x_i| <|u_i|\) for all \(i \in I_+\cup I_-\) or
    \[
    \frac{1}{|I_+|}\sum_{i_+ \in I_+}\mathrm{sign}(u_{i_+})\partial_{i_+} f(x) - \frac{1}{|I_-|}\sum_{i_- \in I_-}\mathrm{sign}(u_{i_-})\partial_{i_-} f(x) \geq \delta.
    \]
\end{lemma}
\begin{proof}
    Fix \(x \in B(x^*,\epsilon) \setminus P\). For any $i,j\in I$, we have $\mathrm{sign}(x_ix_j-u_iu_j) = -\mathrm{sign}(u_iu_j)$ unless either $\{i,j\}\in S_+$ or $\{i,j\}\in S_-$, following from the definition of $\epsilon$. Thus, for any $i_+ \in I \setminus S_-$, the $i_+$-th partial derivative can be expanded from \eqref{eq:partial_exp} as
    \begin{equation}\label{eq:i_+}
        \begin{aligned}
          \partial_{i_+} f(x) &=  \sum_{j\in S_+}\mathrm{sign}(x_{i_+}x_j-u_{i_+}u_j)x_j + \sum_{j\in S_-\cup S_0}\mathrm{sign}(x_{i_+}x_j-u_{i_+}u_j)x_j\\
         &= \sum_{j\in S_+}\mathrm{sign}(x_{i_+}x_j-u_{i_+}u_j)x_j -\sum_{j\in S_-\cup S_0}\mathrm{sign}(u_{i_+}u_j)x_j\\
         &= \sum_{j\in S_+}\mathrm{sign}(x_{i_+}x_j-u_{i_+}u_j)x_j + \mathrm{sign}(u_{i_+})\sum_{j\in S_-}|x_j|  - \mathrm{sign}(u_{i_+})\sum_{j\in S_0}\mathrm{sign}(u_j)x_j.           
        \end{aligned}
    \end{equation}
Similarly, for any $i_- \in I \setminus S_+$, it holds that
\begin{equation}\label{eq:i_-}
    \partial_{i_-} f(x) = \sum_{j\in S_-}\mathrm{sign}(x_{i_-}x_j-u_{i_-}u_j)x_j - \mathrm{sign}(u_{i_-})\sum_{j\in S_+}|x_j|  - \mathrm{sign}(u_{i_-})\sum_{j\in S_0}\mathrm{sign}(u_j)x_j.
\end{equation}
Now fix \(I_+ \subset I \setminus S_-\) and \(I_- \subset I \setminus S_+\). Combining \eqref{eq:i_+} and \eqref{eq:i_-}, We have
 \begin{align*}
     &~\frac{1}{|I_+|}\sum_{i_+ \in I_+}\mathrm{sign}(u_{i_+})\partial_{i_+} f(x) - \frac{1}{|I_-|}\sum_{i_- \in I_-}\mathrm{sign}(u_{i_-})\partial_{i_-} f(x)\\
     =&~\frac{1}{|I_+|}\sum_{i_+ \in I_+}\sum_{j\in S_+}\mathrm{sign}(u_{i_+})\mathrm{sign}(x_{i_+}x_j-u_{i_+}u_j)x_j + \sum_{j\in S_-}|x_j| - \sum_{j\in S_0}\mathrm{sign}(u_j)x_j + \cdots\\
     &-\frac{1}{|I_-|}\sum_{i_- \in I_-}\sum_{j\in S_-}\mathrm{sign}(u_{i_-})\mathrm{sign}(x_{i_-}x_j-u_{i_-}u_j)x_j + \sum_{j\in S_+}|x_j| + \sum_{j\in S_0}\mathrm{sign}(u_j)x_j\\
     =&~\frac{1}{|I_+|}\sum_{i_+ \in I_+}\sum_{j\in S_+}\left(\mathrm{sign}(u_{i_+})\mathrm{sign}(x_{i_+}x_j-u_{i_+}u_j)x_j + |x_j|\right) + \cdots\\
     &+\frac{1}{|I_-|}\sum_{i_- \in I_-}\sum_{j\in S_-}\left(-\mathrm{sign}(u_{i_-})\mathrm{sign}(x_{i_-}x_j-u_{i_-}u_j)x_j + |x_j|\right) \ge 0.
 \end{align*}
This proves the first part of the lemma. To prove the second part, assume that $|x_i|>|u_i|$ for some $i\in I_+\cup I_-$. We will prove the case where $i\in I_+$ and the case where $i\in I_-$ follows verbatim. Since $|x_i|>|u_i|$ with $i\in I_+ \subset I \setminus S_-$, we have $i\in S_+$. Thus,
\begin{align*}
\frac{1}{|I_+|}\sum_{i_+ \in I_+}\sum_{j\in S_+}\left(\mathrm{sign}(u_{i_+})\mathrm{sign}(x_{i_+}x_j-u_{i_+}u_j)x_j + |x_j|\right) &\ge \frac{1}{|I_+|} \left(\mathrm{sign}(u_{i})\mathrm{sign}(x_{i}^2-u_{i}^2)x_i + |x_i|\right)\\
&\ge \frac{2|u_i|}{|I_+|}.
\end{align*}
This proves the second part with $\delta:= 2\min_{i\in I}|u_i|/n>0$.
\end{proof}
Now we continue to prove Theorem 3.1. Consider the following two cases.

\emph{Case 1: both \(S_+\) and  \(S_-\) are nonempty.}
Consider \(g:\mathbb{R}^n \to \mathbb{R}\) defined by
\[
g(x) := \frac{1}{|S_+|}\sum_{i_+ \in S_+}\mathrm{sign}(u_{i_+})x_{i_+}-\frac{1}{|S_-|}\sum_{i_- \in S_-}\mathrm{sign}(u_{i_-})x_{i_-}.
\]
Recall that $x_k\in B(x^*,\epsilon)\setminus P$ for all \(k \ge N_1\). By \eqref{eq:update_simplified}, we have
\begin{small}
 \begin{align*}
    g(x_{k})-g(x_{k+1}) &= \frac{1}{|S_+|}\sum_{i_+ \in S_+}\mathrm{sign}(u_{i_+})(x_{k,i_+}-x_{k+1,i_+})-\frac{1}{|S_-|}\sum_{i_- \in S_-}\mathrm{sign}(u_{i_-})(x_{k,i_-}-x_{k+1,i_-})
    \\ & = \alpha_k\left(\frac{1}{|S_+|}\sum_{i_+ \in S_+}\mathrm{sign}(u_{i_+})\partial_{i_+}f(x_k)-\frac{1}{|S_-|}\sum_{i_- \in S_-}\mathrm{sign}(u_{i_-})\partial_{i_-}f(x_k)\right).
\end{align*}   
\end{small}
By Lemma \ref{lemma:partial}, we know that \(g(x_{k}) \geq g(x_{k+1}) \), and since \(g\) is a linear functional, we know that \(\lim_{k\to \infty}g(x_k) = g(x^*)\). Therefore, for any \(k \ge N_1\), \( g(x_k) \geq g(x^*)\). We next show that
\begin{equation}\label{eq:g_dif}
    g(x_{k}) - g(x_{k+1}) \geq \alpha_k \delta 
\end{equation}
for all $k \ge N_1$. Assume the contrary that above inequality is violated by some $k \ge N_1$, then again by Lemma \ref{lemma:partial}, it holds that \(|x_{k,i}| < |u_i|  = |x^*_i|\) for all \(i \in S_+ \cup S_-\). Thus,
\begin{align*}
    g(x_k) - g(x^*) &= \frac{1}{|S_+|}\sum_{i_+ \in S_+}\mathrm{sign}(u_{i_+})(x_{k,i_+}-x^*_{i_+})-\frac{1}{|S_-|}\sum_{i_- \in S_-}\mathrm{sign}(u_{i_-})(x_{k,i_-}-x^*_{i_-}) 
    \\& = -\frac{1}{|S_+|}\sum_{i_+ \in S_+}|x_{k,i_+}-x^*_{i_+}|-\frac{1}{|S_-|}\sum_{i_- \in S_-}|x_{k,i_-}-x^*_{i_-}|<0.
\end{align*}
This is in contradiction with $g(x_k)\ge g(x^*)$, so \eqref{eq:g_dif} holds. Since \((\alpha_k)_{k\in \mathbb{N}}\) is nonsummable, we have
\[g(x_{N_1}) - g(x_K) = \sum_{k = N_1}^{K-1} g(x_{k}) - g(x_{k+1}) \ge \sum_{k = N_1}^{K-1} \alpha_k \delta \to \infty\]
as $K\to \infty$. This contradicts with $g(x_k)\to g(x^*)$. Therefore, $x_k \not\to x^*$.

\emph{Case 2: There is exactly one empty set among \(S_+\) and \(S_-\).} 
By symmetry, we assume \(S_+ \neq \emptyset\)  and \(|S_-| = \emptyset\) without loss of generality. When \(k \ge N_1\), for any \(j \notin S_+\), \(|x_{k,j}| < |u_j|\). Similar to the previous case, consider \(g_j: \mathbb{R}^n \to \mathbb{R}\) defined by
\[
g_j (x) := \frac{1}{|S_+|}\sum_{i_+ \in S_+}\mathrm{sign}(u_{i_+})x_{i_+}-\mathrm{sign}(u_{j})x_{j}.
\]
We have
\begin{align*}
    g_j(x_{k})-g_j(x_{k+1}) &= \frac{1}{|S_+|}\sum_{i_+ \in S_+}\mathrm{sign}(u_{i_+})(x_{k,i_+}-x_{k+1,i_+})-\mathrm{sign}(u_{j})(x_{k,j}-x_{k+1,j})
    \\ & = \alpha_k \left(\frac{1}{|S_+|}\sum_{i_+ \in S_+}\mathrm{sign}(u_{i_+})\partial_{i_+}f(x_k)-\mathrm{sign}(u_{j})\partial_{j}f(x_k)\right).
\end{align*}
By Lemma \ref{lemma:partial}, we know that \(g_j(x_{k}) \geq g_j(x_{k+1}) \) and \( g_j(x_k) \geq g_j(x^*)\) for all $k\ge N_1$ and $j\not\in S_+$. Furthermore, since \((\alpha_k)_{k \in \mathbb{N}}\) is nonsummable, by applying arguments similar to the previous case to the quantity \(\sum_{j \in I\setminus S_+}g_j(x_k)\), there exists \(N_2 > N_1\) such that \(g_j(x_{N_2})=g_j(x_{N_2+1})\) for all \(j \notin S_+\).  Again by Lemma \ref{lemma:partial}, we have \(|x_{N_2,i_+}| <|u_{i_+}|=|x^*_{i+}|\).

We next prove by induction that for any \(k' \geq N_2\) and \(i_+ \in S_+\),  \(|x_{k'+1,i_+} | < |x_{k',i_+} |< |u_{i_+}| = |x^*_{i_+}|\), which clearly contradicts with $x_k\to x^*$. Suppose for some \(k' \geq N_2\), we have \(|x_{k',i_+} |< |x^*_{i_+}|\) for all \(i_+ \in S_+\). Then, for all \(i_+ \in S_+\) and any \(i \in I\), we have 
\[
\mathrm{sign}(x_{k',i_+}x_{k',i} - u_{i_+}u_i)= \mathrm{sign}(- u_{i_+}u_i) = -\mathrm{sign}(u_{i_+})\mathrm{sign}(u_i),
\]
where we used the fact that $|x_{k',i}|<|u_i|$ for all $i\not\in S_+$, by the definition of $\epsilon$ and the fact that $x_{k'}\in B(x^*,\epsilon)$. Thus,
\begin{align*}
    |x_{k'+1,i_+}| - |x_{k',i_+}|&= \mathrm{sign}(x_{k'+1,i_+}) x_{k'+1,i_+} - \mathrm{sign}(x_{k',i_+}) x_{k',i_+}\\
    &= \mathrm{sign}(u_{i_+})(x_{k'+1,i_+} - x_{k',i_+})\\
    & =  - \mathrm{sign}(u_{i_+})\alpha_{k'} \partial_{i_+}f(x_{k'})\\ 
    & = - \alpha_{k'} \mathrm{sign}(u_{i_+})\sum_{i=1}^n\mathrm{sign}(x_{k',i_+}x_{k',i} - u_{i_+}u_i)x_{k',i}
    \\ &= \alpha_{k'} \sum_{i=1}^n\mathrm{sign}(u_i)x_{k',i}.
\end{align*}
It remains to prove that $\sum_{i=1}^n\mathrm{sign}(u_i)x_{k',i}<0$. Recall that $x^*\in A$, so \(\sum_{i=1}^n\mathrm{sign}(u_i)x^*_{i} = 0\). Therefore,
\begin{align}
    \sum_{i=1}^n  \mathrm{sign}(u_i)x_{k',i} &= \sum_{i=1}^n  \mathrm{sign}(u_i) (x_{k',i} - x^*_{i})\nonumber
    \\& = \sum_{i_+ \in S_+}  \mathrm{sign}(u_i)(x_{k',i_+}- x^*_{i_+}) + \sum_{j \notin S_+}  \mathrm{sign}(u_j)(x_{k',j}-x^*_j). \label{eq:ux_split}
\end{align}
For the first term in \eqref{eq:ux_split}, we have
\begin{equation*}
    \sum_{i_+ \in S_+}\mathrm{sign}(u_{i_+})(x_{k',i_+}-x^*_{i_+}) = -\sum_{i_+ \in S_+}|x_{k',i_+}-x^*_{i_+}| < 0
\end{equation*}
To show that the second term in \eqref{eq:ux_split} is negative as well, note that
\begin{align*}
    0 \leq g_j(x_{k'})-g_j(x^*) = \frac{1}{|S_+|}\sum_{i_+ \in S_+}\mathrm{sign}(u_{i_+})(x_{k',i_+}-x^*_{i_+})-\mathrm{sign}(u_{j})(x_{k',j}-x^*_{j})
\end{align*}
for all \(j \notin S_+\). Thus,
\begin{equation*}
     \mathrm{sign}(u_{j})(x_{k',j}-x^*_{j}) \le \frac{1}{|S_+|}\sum_{i_+ \in S_+}\mathrm{sign}(u_{i_+})(x_{k',i_+}-x^*_{i_+})<0.
\end{equation*}
Therefore, $|x_{k'+1,i_+} | < |x_{k',i_+} |< |u_{i_+}| = |x_{i_+}^*|$. This completes the induction and the proof of Theorem \ref{thm:sym_avoid}. \qed

\subsection{Proof of Proposition \ref{prop:u_i_zero}}\label{sec:proof_u_i_zero}
We prove the proposition by showing that any such subgradient sequence $(x_k)_{k\in \mathbb N}$ stays in \(T\) by induction. Assume that \(x_k \in T\). We have
    \begin{align*}
        x_{k+1,1} & \in x_{k,1}- \alpha_k \partial_1 f(x_k)
        \\& = x_{k,1}- \alpha_k \sum_{j=1}^n \text{sign}(x_{k,1} x_{k,j})x_{k,j}
        \\& = x_{k,1}-\alpha_k \text{sign}(x_{k,1}) \sum_{j=1}^n |x_{k,j}|.
    \end{align*}
    Since \(x_k \in T\), we have
    \begin{equation}\label{eq:contraction}
    |x_{k,1}|\left(1-\alpha_k \left(1+\frac{n-1}{n+1}\right)\right) < |x_{k+1,1}| \leq |x_{k,1}|(1-\alpha_k).        
    \end{equation}
    For \(i\in I \setminus \{1\}\), we have
    \begin{align*}
        x_{k+1,i} & \in x_{k,i}- \alpha_k \partial_i f(x_k)
        \\ &= x_{k,i}- \alpha_k \sum_{j=1}^n \text{sign}(x_{k,i} x_{k,j}-u_i u_j) x_{k,j}.
        \\& = x_{k,i} - \alpha_k (\text{sign}(x_{k,i})|x_{k,1}|+\sum_{j=2}^n \text{sign}(x_{k,i} x_{k,j}-u_i u_j) x_{k,j}).
    \end{align*}
    Since \(x_k \in T\), we have
    \[|x_{k+1,i}| \leq \max\left\{|x_{k,i}| - \alpha_k |x_{k,1}|+\alpha_k \frac{n-1}{n+1}|x_{k,1}|,\alpha_k |x_{k,1}|- |x_{k,i}| +\alpha_k \frac{n-1}{n+1}|x_{k,1}|\right\}.\]
    Note that
    \begin{align*}
        &\frac{|x_{k+1,1}|}{n+1}-\left(|x_{k,i}| - \alpha_k |x_{k,1}|+\alpha_k \frac{n-1}{n+1}|x_{k,1}|\right) \\
        >~&\frac{1}{n+1}|x_{k,1}|\left(1-\alpha_k\left(1+\frac{n-1}{n+1}\right)\right) - |x_{k,1}|\left(\frac{1}{n+1}- \alpha_k  \frac{2}{n+1}\right)
        \\
        =~& \frac{2\alpha_k |x_{k,1}| }{(n+1)^2} > 0,
    \end{align*}
    and 
    \begin{align*}
        &\frac{|x_{k+1,1}|}{n+1}-\left(\alpha_k |x_{k,1}|- |x_{k,i}| +\alpha_k \frac{n-1}{n+1}|x_{k,1}|\right)\\
        >~& |x_{k,1}|\left(1-\alpha_k\left(1+\frac{n-1}{n+1}\right)\right) -\alpha_k \frac{2n}{n+1}|x_{k,1}|
        \\=~& |x_{k,1}|\left(1-\alpha_k \frac{4n}{n+1}\right)
        \\\geq~& \frac{|x_{k,1}|}{n+1} > 0.
    \end{align*}
    Therefore, \(|x_{k+1,i}|<\frac{|x_{k+1,1}|}{n+1}\) and thus \(x_{k+1} \in T\). This completes the induction, and we have $x_k\in T$ for all $k\in \mathbb N$.
    
    To show that $x_k\to 0$, by \eqref{eq:contraction}, we have \[
    |x_{k,1}| \leq |x_{0,1}|\prod_{l=0}^{k-1}(1-\alpha_l)\to 0 
    \]
    as \((\alpha_k)_{k \in \mathbb{N}}\) is nonsummable. Since \(x_k\in T\) for all \(k \in \mathbb{N}\), we can conclude that \((x_k)_{k \in \mathbb{N}}\) converges to \(0\). \qed

\appendix
\section{Omitted proofs}
\subsection{Proof of Proposition \ref{prop:phase}}\label{sec:proof_phase}
For each $i = 1,\ldots,N$, denote by $A_i := a_i a_i^\top$, which is a positive semidefinite matrix. Define the subspace:
\[
V = \mathrm{span}\{ a_1, a_2, \dots, a_N \} \subset\mathbb{R}^d\]
Let $V^\perp$ be the orthogonal complement of $V$. By \cite[2.3.10 Chain Rule II]{clarke1990optimization}, the subdifferential of $f$ is given by
\[\partial f(x) = \frac{1}{N}\sum_{i = 1}^N \langle a_i,x\rangle\mathrm{sign}(|\langle a_i,x\rangle|^2 - b_i)a_i\subset V.\]

Let $x:[0,\infty)\to\mathbb R^{n}$ be a subgradient trajectory of $f$, then it satisfies
\[
x'(t) \in -\partial f(x(t)) \subset -V = V, \quad \text{for a.e. } t>0.
\]
Decompose the trajectory as:
\[
x(t) = x_V(t) + x_{V^\perp}(t),
\]
where $x_V(t) \in V$ and $x_{V^\perp}(t) \in V^\perp$. Since $x'(t) \in V$ for almost every $t$, we have:
\[
x_{V^\perp}'(t) = 0 \implies x_{V^\perp}(t) = x_{V^\perp}(0)\text{ for all } t>0.
\]
Thus, \(x_{V^\perp}(t)\) is a constant and $x(t) = x_V(t) + x_{V^\perp}(0)$. It remains to prove that $x_V(\cdot)$ is bounded. By Lemma \ref{lemma:chain}, the function value is nonincreasing along any subgradient trajectory. Thus,
\[
f(x(t)) \leq f(x(0)) \quad \forall t \geq 0.
\]
Let $C_0 = f(x(0))$. From the definition of $f$ \eqref{eq:phase}, we have
\[
f(x) = \frac{1}{N} \sum_{i=1}^N \left| |\langle a_i,x \rangle|^2 - b_i \right| \geq \frac{1}{N} \sum_{i=1}^N \left( |\langle a_i,x \rangle|^2 - |b_i| \right) = \frac{1}{N} x^\top B x - \frac{1}{N} \sum_{i=1}^N |b_i|,
\]
where $B = \sum_{i=1}^N a_i a_i^\top = \sum_{i=1}^N A_i$. Combining with $f(x(t)) \leq C_0$, we have
\[
x(t)^\top B x(t) \leq N C_0 + \sum_{i=1}^N |b_i| := C_1.
\]
Since $B = \sum_{i = 1}^N a_i a_i^\top$, $\mathrm{Null}(B) = V^\perp$. On $V$, $B$ is positive definite. Thus, there exists $\lambda^+ > 0$ such that
    \[
   x(t)^\top B x(t) =  x_V(t)^\top B x_V(t) \geq \lambda^+ \|x_V(t)\|^2
    \]
Then,
    \[
    \lambda^+ \|x_V(t)\|^2 \leq C_1 \implies \|x_V(t)\|^2 \leq \frac{C_1}{\lambda^+}
    \]
Therefore, the subgradient trajectory $x(\cdot)$ is bounded. \qed
\subsection{Proof of Proposition \ref{prop:sensing}} \label{sec:proof_sensing}
Let \( Z : [0, \infty) \to \mathbb{R}^{m \times r} \times \mathbb{R}^{n \times r} \) be a subgradient trajectory of $f$, which satisfies
\[
Z'(t) \in -\partial f(Z(t)), \quad \text{for almost every } t \geq 0, \quad \text{and } Z(0) = (X_0, Y_0).
\]
By \cite[2.3.10 Chain Rule II]{clarke1990optimization}, the subdifferential of $f$ is given by
\[\partial f(X,Y):= \left\{ \left(\frac{1}{N}\sum_{i=1}^Nc_iA_iY, \frac{1}{N}\sum_{i=1}^Nc_iA_i^\top X\right): c_i \in \mathrm{sign}(\langle A_i,XY^\top\rangle - b_i)\right\}.\]
Hence, with \( Z := (X, Y) \), for almost every \( t \geq 0 \) we have
\[X'(t) = -\frac{1}{N}\sum_{i=1}^Nc_iA_iY, \quad Y'(t) = -\frac{1}{N}\sum_{i=1}^Nc_iA_i^\top X,\quad\text{for some }c_i\in \mathrm{sign}(\langle A_i,X(t)Y(t)^\top\rangle - b_i).\]
Consider \( \phi : [0, \infty) \to \mathbb{R}^{r\times r} \) defined by \( \phi(t) := X(t)^\top X(t) - Y(t)Y(t)^\top \), whose derivative is given by

\begin{equation}
\phi'(t) = X'(t)^\top X(t) + X(t)^\top X'(t) - Y'(t)Y(t)^\top - Y(t)Y'(t)^\top = 0.
\end{equation}

Therefore, the continuous function \( \phi \) is constant on \([0, \infty)\). Thus, for  any $t\ge 0$, it holds that
\begin{align*}
c\|X(t)Y(t)^\top\|_F^2 &\le \frac{1}{N}\sum_{i=1}^N\langle A_i,X(t)Y(t)^\top\rangle^2\\
&\le \frac{1}{N}\left(\sum_{i=1}^N|\langle A_i,X(t)Y(t)^\top\rangle|\right)^{2} \\
&\leq \frac{1}{N}\left(\sum_{i=1}^N(|\langle A_i,X(t)Y(t)^\top\rangle-b_i|+|b_i|)\right)^{2} \\
&= \frac{1}{N}\left(\sum_{i=1}^N|\langle A_i,X(t)Y(t)^\top\rangle-b_i|+\sum_{i=1}^N |b_i|\right)^{2}\\
&=\frac{1}{N}\left(Nf(X(t),Y(t))+\sum_{i=1}^N |b_i|\right)^{2}\\
&\leq \frac{1}{N}\left(Nf(X(0),Y(0))+\sum_{i=1}^N |b_i|\right)^{2}=:c_1,
\end{align*}
where we used the fact that function values are positive and decrease along any subgradient trajectory of $f$ by Lemma \ref{lemma:chain}. Hence \(\|X(t)Y(t)^T\|_F^2 \leq c_2 := c_1/c\). Notice that
\begin{align*}
&\|X(t)^\top X(t)\|_F^2 + \|Y(t)^\top Y(t)\|_F^2 \\
= &\|X(t)^\top X(t) - Y(t)^\top Y(t)\|_F^2 + 2\|X(t)Y(t)^\top\|_F^2 \\
\leq& \|X(0)^\top X(0) - Y(0)^\top Y(0)\|_F^2 + 2c_2\\
=:&c_3
\end{align*}
By the Cauchy-Schwarz inequality,
\[
\|X(t)\|_F^4 + \|Y(t)\|_F^4 \leq \operatorname{rank}(X)\|X(t)^\top X(t)\|_F^2 + \operatorname{rank}(Y^\top)\|Y(t)^\top Y(t)\|_F^2 \leq (m + n + r)c_3.
\]
Thus, the subgradient trajectory $Z(\cdot)$ is bounded. \qed

\bibliographystyle{alpha}
\bibliography{mybib}

@article{bolte2020mathematical,
  title={A mathematical model for automatic differentiation in machine learning},
  author={Bolte, J{\'e}r{\^o}me and Pauwels, Edouard},
  journal={Advances in Neural Information Processing Systems},
  volume={33},
  pages={10809--10819},
  year={2020}
}

@book{lai2024stability,
  title={Stability of first-order methods in tame optimization},
  author={Lai, Lexiao},
  year={2024},
  publisher={Columbia University}
}

@article{huber1964robust,
  author  = {Huber, Peter J.},
  title   = {Robust Estimation of a Location Parameter},
  journal = {The Annals of Mathematical Statistics},
  volume  = {35},
  number  = {1},
  pages   = {73--101},
  year    = {1964},
  doi     = {10.1214/aoms/1177703732}
}

@book{hampel1986robust,
  author    = {Hampel, Frank R. and Ronchetti, Elvezio M. and Rousseeuw, Peter J. and Stahel, Werner A.},
  title     = {Robust Statistics: The Approach Based on Influence Functions},
  publisher = {John Wiley \& Sons},
  year      = {1986},
  isbn      = {9780471829218}
}

@article{koenker1978regression,
  author  = {Koenker, Roger and Bassett, Gilbert},
  title   = {Regression Quantiles},
  journal = {Econometrica},
  volume  = {46},
  number  = {1},
  pages   = {33--50},
  year    = {1978},
  doi     = {10.2307/1913643}
}

@article{donoho2006compressed,
  author  = {Donoho, David L.},
  title   = {Compressed Sensing},
  journal = {IEEE Transactions on Information Theory},
  volume  = {52},
  number  = {4},
  pages   = {1289--1306},
  year    = {2006},
  doi     = {10.1109/TIT.2006.871582}
}

@article{candes2006stable,
  author  = {Cand{\`e}s, Emmanuel J. and Romberg, Justin and Tao, Terence},
  title   = {Stable Signal Recovery from Incomplete and Inaccurate Measurements},
  journal = {Communications on Pure and Applied Mathematics},
  volume  = {59},
  number  = {8},
  pages   = {1207--1223},
  year    = {2006},
  doi     = {10.1002/cpa.20124}
}

@article{josz2025nonsmooth,
  title={Nonsmooth rank-one symmetric matrix factorization landscape},
  author={Josz, C{\'e}dric and Lai, Lexiao},
  journal={Optimization Letters},
  pages={1--8},
  year={2025},
  publisher={Springer}
}

@article{guan2024ell_1,
  title={$\ell_1$-norm rank-one symmetric matrix factorization has no spurious second-order stationary points},
  author={Guan, Jiewen and So, Anthony Man-Cho},
  journal={arXiv preprint arXiv:2410.05025},
  year={2024}
}

@article{recht2010guaranteed,
  title={Guaranteed minimum-rank solutions of linear matrix equations via nuclear norm minimization},
  author={Recht, Benjamin and Fazel, Maryam and Parrilo, Pablo A},
  journal={SIAM review},
  volume={52},
  number={3},
  pages={471--501},
  year={2010},
  publisher={SIAM}
}

@article{fougereux2024global,
  title={Global convergence of gradient descent for phase retrieval},
  author={Fougereux, Th{\'e}odore and Josz, C{\'e}dric and Li, Xiaopeng},
  journal={arXiv preprint arXiv:2410.09990},
  year={2024}
}

@article{ma2023global,
  title={Global convergence of sub-gradient method for robust matrix recovery: Small initialization, noisy measurements, and over-parameterization},
  author={Ma, Jianhao and Fattahi, Salar},
  journal={Journal of Machine Learning Research},
  volume={24},
  number={96},
  pages={1--84},
  year={2023}
}

@article{eldar2014phase,
  title={Phase retrieval: Stability and recovery guarantees},
  author={Eldar, Yonina C and Mendelson, Shahar},
  journal={Applied and Computational Harmonic Analysis},
  volume={36},
  number={3},
  pages={473--494},
  year={2014},
  publisher={Elsevier}
}

@inproceedings{meng2013cyclic,
  title={A cyclic weighted median method for l1 low-rank matrix factorization with missing entries},
  author={Meng, Deyu and Xu, Zongben and Zhang, Lei and Zhao, Ji},
  booktitle={Proceedings of the AAAI Conference on Artificial Intelligence},
  volume={27},
  number={1},
  pages={704--710},
  year={2013}
}

@inproceedings{eriksson2010efficient,
  title={Efficient computation of robust low-rank matrix approximations in the presence of missing data using the L1 norm},
  author={Eriksson, Anders and Van Den Hengel, Anton},
  booktitle={2010 IEEE Computer society conference on computer vision and pattern recognition},
  pages={771--778},
  year={2010},
  organization={IEEE}
}

@inproceedings{ke2005robust,
  title={Robust $L_1$ norm factorization in the presence of outliers and missing data by alternative convex programming},
  author={Ke, Qifa and Kanade, Takeo},
  booktitle={2005 IEEE Computer Society Conference on Computer Vision and Pattern Recognition (CVPR'05)},
  volume={1},
  pages={739--746},
  year={2005},
  organization={IEEE}
}

@article{basri2003lambertian,
  title={Lambertian reflectance and linear subspaces},
  author={Basri, Ronen and Jacobs, David W},
  journal={IEEE transactions on pattern analysis and machine intelligence},
  volume={25},
  number={2},
  pages={218--233},
  year={2003},
  publisher={IEEE}
}

@article{santambrogio2017euclidean,
  title={$\{$Euclidean, metric, and Wasserstein$\}$ gradient flows: an overview},
  author={Santambrogio, Filippo},
  journal={Bulletin of Mathematical Sciences},
  volume={7},
  pages={87--154},
  year={2017},
  publisher={Springer}
}

@article{davis2025active,
  title={Active manifolds, stratifications, and convergence to local minima in nonsmooth optimization},
  author={Davis, Damek and Drusvyatskiy, Dmitriy and Jiang, Liwei},
  journal={Foundations of Computational Mathematics},
  pages={1--83},
  year={2025},
  publisher={Springer}
}

@article{li2024revisiting,
  title={Revisiting subgradient method: Complexity and convergence beyond Lipschitz continuity},
  author={Li, Xiao and Zhao, Lei and Zhu, Daoli and So, Anthony Man-Cho},
  journal={Vietnam Journal of Mathematics},
  pages={1--21},
  year={2024},
  publisher={Springer}
}

@article{li2020nonconvex,
  title={Nonconvex robust low-rank matrix recovery},
  author={Li, Xiao and Zhu, Zhihui and Man-Cho So, Anthony and Vidal, Rene},
  journal={SIAM Journal on Optimization},
  volume={30},
  number={1},
  pages={660--686},
  year={2020},
  publisher={SIAM}
}

@article{josz2024proximal,
  title={Proximal random reshuffling under local Lipschitz continuity},
  author={Josz, Cedric and Lai, Lexiao and Li, Xiaopeng},
  journal={arXiv preprint arXiv:2408.07182v1},
  year={2024}
}

@article{lewis2002active,
  title={Active sets, nonsmoothness, and sensitivity},
  author={Lewis, Adrian S},
  journal={SIAM Journal on Optimization},
  volume={13},
  number={3},
  pages={702--725},
  year={2002},
  publisher={SIAM}
}

@article{josz2024sufficient,
  title={Sufficient conditions for instability of the subgradient method with constant step size},
  author={Josz, C{\'e}dric and Lai, Lexiao},
  journal={SIAM Journal on Optimization},
  volume={34},
  number={1},
  pages={57--70},
  year={2024},
  publisher={SIAM}
}

@article{bianchi2023stochastic,
  title={Stochastic subgradient descent escapes active strict saddles on weakly convex functions},
  author={Bianchi, Pascal and Hachem, Walid and Schechtman, Sholom},
  journal={Mathematics of Operations Research},
  year={2023},
  publisher={INFORMS}
}

@article{josz2022nonsmooth,
  title={Nonsmooth rank-one matrix factorization landscape},
  author={Josz, C{\'e}dric and Lai, Lexiao},
  journal={Optimization Letters},
  volume={16},
  number={6},
  pages={1611--1631},
  year={2022},
  publisher={Springer}
}

@article{josz2023convergence,
  title={Convergence of the momentum method for semialgebraic functions with locally Lipschitz gradients},
  author={Josz, C{\'e}dric and Lai, Lexiao and Li, Xiaopeng},
  journal={SIAM Journal on Optimization},
  volume={33},
  number={4},
  pages={3012--3037},
  year={2023},
  publisher={SIAM}
}

@article{bolte2025inexact,
  title={Inexact subgradient methods for semialgebraic functions},
  author={Bolte, J{\'e}r{\^o}me and Le, Tam and Moulines, Eric and Pauwels, Edouard},
  journal={Mathematical Programming},
  pages={1--27},
  year={2025},
  publisher={Springer}
}

@article{josz2023globalstability,
  title={Global stability of first-order methods for coercive tame functions},
  author={Josz, C{\'e}dric and Lai, Lexiao},
  journal={Mathematical Programming},
  pages={1--26},
  year={2023},
  publisher={Springer}
}

@article{ioffe2009,
author = {Ioffe, A. D.},
title = {An Invitation to Tame Optimization},
journal = {SIAM Journal on Optimization},
volume = {19},
number = {4},
pages = {1894-1917},
year = {2009}
}

@article{josz2023global,
  title={Global convergence of the gradient method for functions definable in o-minimal structures},
  author={Josz, C{\'e}dric},
  journal={Mathematical Programming},
  pages={1--29},
  year={2023},
  publisher={Springer}
}

@article{davis2019stochastic2,
  title={Stochastic model-based minimization of weakly convex functions},
  author={Davis, Damek and Drusvyatskiy, Dmitriy},
  journal={SIAM Journal on Optimization},
  volume={29},
  number={1},
  pages={207--239},
  year={2019},
  publisher={SIAM}
}

@inproceedings{kurdyka1998gradients,
  title={On gradients of functions definable in o-minimal structures},
  author={Kurdyka, Krzysztof},
  booktitle={Annales de l'institut Fourier},
  volume={48},
  pages={769--783},
  year={1998}
}

@article{lojasiewicz1982trajectoires,
  title={Sur les trajectoires du gradient d’une fonction analytique},
  author={\L{}ojasiewicz, Stanislaw},
  journal={Seminari di geometria},
  volume={1983},
  pages={115--117},
  year={1982}
}

@article{shor1962application,
  title={Application of the Gradient Method for the Solution of Network Transportation Problems. fJotes, Scientific seminar on theory and applications of cybernetics and operations research},
  author={Shor, NZ},
  journal={Kiev: Academy of Sciences USSR},
  year={1962}
}

@article{lojasiewicz1959,
  title={Sur le problème de la division},
  author={S. \L{}ojasiewicz},
  journal={Studia Mathematica},
  pages={87–136},
  year={1959}
}

@article{bianchi2022convergence,
  title={Convergence of constant step stochastic gradient descent for non-smooth non-convex functions},
  author={Bianchi, Pascal and Hachem, Walid and Schechtman, Sholom},
  journal={Set-Valued and Variational Analysis},
  pages={1--31},
  year={2022},
  publisher={Springer}
}

@article{charisopoulos2021low,
  title={Low-rank matrix recovery with composite optimization: good conditioning and rapid convergence},
  author={Charisopoulos, Vasileios and Chen, Yudong and Davis, Damek and D{\'\i}az, Mateo and Ding, Lijun and Drusvyatskiy, Dmitriy},
  journal={Foundations of Computational Mathematics},
  pages={1--89},
  year={2021},
  publisher={Springer}
}

@article{bolte2007clarke,
  title={Clarke subgradients of stratifiable functions},
  author={Bolte, J{\'e}r{\^o}me and Daniilidis, Aris and Lewis, Adrian and Shiota, Masahiro},
  journal={SIAM Journal on Optimization},
  volume={18},
  number={2},
  pages={556--572},
  year={2007},
  publisher={SIAM}
}

@article{daniilidis2020pathological,
  title={Pathological subgradient dynamics},
  author={Daniilidis, Aris and Drusvyatskiy, Dmitriy},
  journal={SIAM Journal on Optimization},
  volume={30},
  number={2},
  pages={1327--1338},
  year={2020},
  publisher={SIAM}
}

@article{drusvyatskiy2015curves,
  title={Curves of descent},
  author={Drusvyatskiy, Dmitriy and Ioffe, Alexander D and Lewis, Adrian S},
  journal={SIAM Journal on Control and Optimization},
  volume={53},
  number={1},
  pages={114--138},
  year={2015},
  publisher={SIAM}
}

@article{tarski1951decision,
  title={A decision method for elementary algebra and geometry: Prepared for publication with the assistance of JCC McKinsey},
  author={Tarski, Alfred},
  year={1951},
  publisher={Rand Corporation}
}

@article{duchi2019solving,
  title={Solving (most) of a set of quadratic equalities: Composite optimization for robust phase retrieval},
  author={Duchi, John C and Ruan, Feng},
  journal={Information and Inference: A Journal of the IMA},
  volume={8},
  number={3},
  pages={471--529},
  year={2019},
  publisher={Oxford University Press}
}

@article{davis2020nonsmooth,
  title={The nonsmooth landscape of phase retrieval},
  author={Davis, Damek and Drusvyatskiy, Dmitriy and Paquette, Courtney},
  journal={IMA Journal of Numerical Analysis},
  volume={40},
  number={4},
  pages={2652--2695},
  year={2020},
  publisher={Oxford University Press}
}

@inproceedings{bennett2007netflix,
  title={The netflix prize},
  author={Bennett, James and Lanning, Stan and others},
  booktitle={Proceedings of KDD cup and workshop},
  volume={2007},
  pages={35},
  year={2007},
  organization={New York}
}

@article{candes2007sparsity,
  title={Sparsity and incoherence in compressive sampling},
  author={Candes, Emmanuel and Romberg, Justin},
  journal={Inverse problems},
  volume={23},
  number={3},
  pages={969},
  year={2007},
  publisher={IOP Publishing}
}

@article{gillis2018complexity,
  title={{On the complexity of robust PCA and l1-norm low-rank matrix approximation}},
  author={Gillis, Nicolas and Vavasis, Stephen A},
  journal={Mathematics of Operations Research},
  volume={43},
  number={4},
  pages={1072--1084},
  year={2018},
  publisher={INFORMS}
}

@article{duchi2018stochastic,
  title={Stochastic methods for composite and weakly convex optimization problems},
  author={Duchi, John C and Ruan, Feng},
  journal={SIAM Journal on Optimization},
  volume={28},
  number={4},
  pages={3229--3259},
  year={2018},
  publisher={SIAM}
}

@article{davis2020stochastic,
  title={Stochastic subgradient method converges on tame functions},
  author={Davis, Damek and Drusvyatskiy, Dmitriy and Kakade, Sham and Lee, Jason D},
  journal={Foundations of computational mathematics},
  volume={20},
  number={1},
  pages={119--154},
  year={2020},
  publisher={Springer}
}

@article{bouwmans2019,
  title={Background subtraction in real applications: Challenges, current models and future directions},
  author={T. Bouwmans and B. Garcia-Garcia},
  journal={arXiv preprint arXiv:1901.03577},
  year={2019}
}

@article{yuan2013sparse,
  title={Sparse and low-rank matrix decomposition via alternating direction methods},
  author={Xiaoming Yuan and Junfeng Yang},
  journal={Pacific Journal of Optimization},
  year={2013}
}

@article{lin2011, 
	author={Z. Lin and R. Liu and Z. Su},
	title={{Linearized Alternating Direction Method with Adaptive Penalty for Low-Rank Representation}}, 
	journal={NeurIPS},
	year={2011}
}

@article{candes2011,
  title={Robust principal component analysis?},
  author={E. J. Cand{\`e}s and X. Li and Y. Ma and J. Wright},
  journal={Journal of the ACM (JACM)},
  volume={58},
  number={3},
  pages={1--37},
  year={2011},
  publisher={ACM New York, NY, USA}
}

@article{joszneurips2018, 
	author={C. Josz and Y. Ouyang and R. Y. Zhang and J. Lavaei and S. Sojoudi},  
journal={NeurIPS},
	title={{A theory on the absence of spurious solutions for nonconvex and nonsmooth optimization}}, 
	month={Dec.},
	year={2018}
}

@incollection{lrslibrary2015,
author    = {Sobral, A. and Bouwmans, T. and Zahzah, E.},
title     = {LRSLibrary: Low-Rank and Sparse tools for Background Modeling and Subtraction in Videos},
booktitle = {Robust Low-Rank and Sparse Matrix Decomposition: Applications in Image and Video Processing},
publisher = {CRC Press, Taylor and Francis Group.},
year      = {2015}
}

@article{burer2003,
author = {S. Burer and R. D.C. Monteiro},
title = {{A nonlinear programming algorithm for solving semidefinite programs via low-rank factorization}},
journal = {Mathematical Programming},
volume = {95},
issue = {2},
pages = {329--357},
year = {2003}
}

@article{clarke1975,
author = {F. H. Clarke},
title = {{Generalized gradients and applications}},
journal = {Transactions of the American Mathematical Society},
year = {1975}
}

@article{bolte2022long,
  title={Long term dynamics of the subgradient method for Lipschitz path differentiable functions},
  author={Bolte, J{\'e}r{\^o}me and Pauwels, Edouard and R{\'\i}os-Zertuche, Rodolfo},
  journal={Journal of the European Mathematical Society},
  year={2022}
}

@article{lai2025diameter,
  title={On the diameter of subgradient sequences in o-minimal structures},
  author={Lai, Lexiao and Song, Mingzhi},
  journal={arXiv preprint arXiv:2511.06868v1},
  year={2025}
}

@book{clarke1990optimization,
  title={Optimization and nonsmooth analysis},
  author={Clarke, Frank H},
  year={1990},
  publisher={SIAM}
}

@article{josz2023certifying,
  title={Certifying the absence of spurious local minima at infinity},
  author={Josz, C{\'e}dric and Li, Xiaopeng},
  journal={SIAM Journal on Optimization},
  volume={33},
  number={3},
  pages={1416--1439},
  year={2023},
  publisher={SIAM}
}
\end{document}